\documentclass{article}
\usepackage[utf8]{inputenc}
\usepackage[T1]{fontenc}
\usepackage[english]{babel}
\usepackage{amsmath,amssymb,amsfonts,amsthm,mathptmx}
\usepackage[inner=3.4cm,outer=3.65cm,bottom=3cm]{geometry}

\usepackage{hyperref}
\usepackage{color}

\theoremstyle{remark}
\newtheorem{remark}{Remark}[section]
\theoremstyle{definition}
\newtheorem{theorem}{Theorem}[section]
\newtheorem{definition}[theorem]{Definition}

\newtheorem{proposition}[theorem]{Proposition}
\newtheorem{lemma}[theorem]{Lemma}

\newtheorem{hypothesis}[theorem]{Hypothesis}
\DeclareMathOperator{\R}{\mathbb{R}}

\DeclareMathOperator{\C}{\mathcal{C}}

\DeclareMathOperator{\N}{\mathbb{N}}
\DeclareMathOperator{\esssup}{ess\,sup}

\DeclareMathOperator{\ra}{\rightarrow}
\DeclareMathOperator{\de}{\text{d}}

\newcommand{\sym}{\text{sym}}

\newcommand{\f}[1]{{\pmb{ #1}}}
\DeclareMathOperator{\di}{\nabla\cdot}

\newcommand{\tet}{\tilde{\theta}}
\newcommand{\tmu}{\tilde{\mu}}
\newcommand{\tp}{\tilde{\varphi}}
\newcommand{\tu}{\tilde{\f u}}
\newcommand{\ov}[1]{{\overline{ #1}}}

\renewcommand{\t}{\partial_t}
\renewcommand{\log}{\ln}

\newcommand{\fhi}{\varphi}

\allowdisplaybreaks
\title{Analysis of a thermodynamically consistent {Navier--Stokes--Cahn--Hilliard }model}
\author{Robert Lasarzik}
\date{Version of the draft: \today}
\begin{document}

\maketitle
\begin{abstract}
In this paper, existence of generalized
solutions to a thermodynamically consistent Navier--Stokes--Cahn--Hilliard model introduced in \cite{giulio} is proven in any space dimension.
The generalized solvability concepts are measure-valued and dissipative solutions. The measure-valued formulation incorporates an entropy inequality and an energy inequality instead of an energy balance in a nowadays standard way, the Gradient flow of the internal variable is fulfilled in a weak  and the momentum balance in a measure-valued sense. 
 In the dissipative formulation, the distributional relations of the momentum balance and the energy as well as entropy inequality are replaced by a relative energy inequality. 
%
%
Additionally, we prove the weak-strong uniqueness of the proposed solution concepts 
and that all generalized solutions with additional regularity are indeed strong solutions. 

\end{abstract}

\tableofcontents
\section{Introduction}
This paper is concerned with the analysis of the initial boundary-value problem for the following PDE system
\begin{subequations}\label{eq}
\begin{align}
\t \f u + ( \f u \cdot \nabla) \f u + \nabla p -\di \left (\nu(\theta) \nabla \f u  \right  ) ={}& -\varepsilon  \di ( \nabla \varphi \otimes\nabla  \varphi ) \,, \quad  \label{eq1}\\
\di \f u ={}& 0 \,, \label{eq4}\\
c_H(\theta) ( \t \theta + ( \f u \cdot \nabla ) \theta ) + \theta ( \t \varphi + (\f u \cdot \nabla) \varphi ) - \di ( \kappa( \theta ) \nabla \theta ) = {}&\nu ( \theta ) | \nabla \f u|^2 + |\nabla \mu|^2    \,, \label{eq3}\\
\t \varphi + ( \f u \cdot \nabla ) \varphi = \Delta \mu\,,  \quad \mu ={}& - \varepsilon\Delta \varphi + \frac{1}{\varepsilon} F'(\varphi ) -\theta \,, \label{eq2}
%
%
%
\end{align}
which describes phase transition phenomena in incompressible fluids. We consider a bounded domain $\Omega\subset \R^d$ or $d\geq 2$ with sufficiently smooth boundary and fix a time interval $[0,T]$. The state variables are the \textit{velocity field}~$\f u$, the \textit{temperature} $\theta $, and the \textit{order parameter} $\varphi$ describing the locally attained phase. 
The pressure is denoted by $p$, the \textit{chemical potential} $\mu$ is mainly an auxiliary variable.  
 The variable $F$ denotes some energy density for the order parameter, $\kappa(\theta) >0$ the  heat conductivity, and $\nu(\theta) >0 $ the viscosity.
 The parameter $\varepsilon$ is related to the interface thickness, which should be small. The variable $c_H$ stands for the specific heat, which is going to be made precise later. 
 For the fluid flow, we choose homogeneous Dirichlet boundary conditions and the other variables are equipped 
 with homogeneous {Neumann} boundary conditions,\textit{ i.e.,}
\begin{align}
\f u = 0 \,, \quad\f n \cdot   \kappa(\theta)  \nabla\theta = 0  \,, \quad \f n \cdot \nabla \varphi = 0 \,, \quad \f n \cdot \nabla \mu = 0\,, \qquad \text{on } \partial \Omega \times (0,T)\,\label{boundary}
\end{align}
and initial conditions
\begin{align}
\f u (0) = \f u_0 \,, \quad \theta(0) = \theta_0 \,, \quad \varphi(0) = \varphi_0 \qquad \text{in } \Omega \,.\label{initial}
\end{align}
\end{subequations}

The system under consideration couples a Navier--Stokes-like equation \eqref{eq1} and~\eqref{eq4} with a Cahn--Hilliard system~\eqref{eq2} and an internal energy balance~\eqref{eq3}. 
On the one hand, this system is interesting in terms of its applications, which reach from modeling of cancer evolution and treatment~\cite{cancer}  over fluid flow for mixtures~\cite{fluidflow,lubo} and for instance modelling of industrial processes like  $3$d printing~\cite{threeDprint}. 
On the other hand, it serves as an interesting prototype of a thermodynamical consistent model of a complex fluid. The energy balance is coupled to  Navier--Stokes-like equations and this in addition to a Gradient Flow for the internal variable. Such systems are omnipresent in applications as well as analysis in the form of two fluid flow, anisotropic fluids like liquid crystals or polymers and additionally, it can be seen as a special form of the so-called GENERIC modeling approach (see~\cite{generic}).

The aim of the article at hand is thus also twofold. On the one hand, we will provide a sound mathematical treatment of this special system, but on the other hand, we 
see this system as a prototype of GENERIC systems and want to infer knowledge on how to define a reasonable solutions concepts for such system. 

\subsection{Review of known results\label{sec:rev}}
The Cahn--Hilliard system originally proposed in~\cite{chanhilliard} received a lot of interest in recent years (see for instance~\cite{miranville}). 
There are lots of works concerning the constant temperature case of~\eqref{eq} (see \textit{e.g.}~\cite{abels,cao,strat} and references therein) and also a recent work on the case of vanishing velocity field (see~\cite{zafferi}). 
Even though there are many publications on the coupled Navier--Stokes Cahn--Hilliard model in the constant temperature case, there are very few publications on the non-isothermal case. 
The considered model~\eqref{eq} was introduced in~\cite{giulio}, where also the existence of weak solutions (fulfilling an additional energy balance) was shown under  additional growth conditions for the heat capacity and the heat conductivity (the assumption are precisely $\delta \in [1/2,1)$ and $\beta\geq 2$, check Hypothesis~\ref{hypo} below for the definition of these parameters). 
In~\cite{euleri} the same authors where able to show existence of weak solutions in the two dimensional case under milder assumptions. 
In~\cite{larosc} a simpler model was considered, an Allen--Cahn type equation with energy balance. But the techniques there are similar to the ones used in this article, namely relying on the \textit{relative energy approach}. 
So far, there is no solution concept available for the relevant case, where $ \kappa(\theta)$ and $c_H(\theta)$ in~\eqref{eq3} are constant. 
In the article at hand, we provide a remedy by considering \textit{measure-valued} and  \textit{dissipative solutions} especially for this case. 

The concept of \textit{Young measure-valued solutions} was first introduced by Tartar~\cite{tartar}. Later on, the concept of \textit{generalized Young measures} was used by DiPerna and Majda~\cite{DiPernaMajda} to define generalized solutions to the Euler equations. These generalized Young measures capture oscillation and concentration effects for sequences bounded in $L^1$.
 Such generalized Young measures have been applied to the complete Euler system~\cite{breit} or the Ericksen--Leslie system equipped with the Oseen--Frank energy~\cite{meas}. 

The concept of a dissipative solution was first introduced by P.-L.~Lions in the context of the Euler equations~\cite[Sec.~4.4]{lionsfluid}, with ideas originating from the Boltzmann equation~\cite{LionsBoltzman}. 
It is also applied in the context of incompressible viscous
electro-magneto-hydro\-dy\-na\-mics (see%
~\cite{raymond})
 and equations of viscoelastic diffusion in polymers~\cite{viscoelsticdiff} as well as liquid cystals~\cite{diss}.
A dissipative solution relies on an appropriate relative energy inequality, which may  be interpreted as a \textit{variation of the energy-dissipation mechanism} of the system. In this solution concept, the momentum conservation is not fulfilled in some distributional sense, but rather in terms of a variation of the underlying energy dissipation principle.

 Beside the fact, that this solution concept complies with the minimum assumptions to a solution concept of existence and weak-strong uniqueness, it appears naturally when considering singular limits~\cite{saint} and in comparisson to measure-valued solutions, it is numerically traceable. In case of anisotropic complex fluids,  in~\cite{approx} the convergence of a semi-discretization was shown and an associated optimal control problem was solved via the dissipative solution concept, whereas in~\cite{nematicelectro} it was proven for the more-involved system describing nematic electrolytes that the solutions to a fully discrete finite element discretization  converge to a dissipative solution in the limit. 
In these works, it was observed that natural discretizations complying with the properties of the system, like energetic or entropic principles, as well as algebraic restrictions converge naturally to a dissipative solution instead of a measure valued solution (see~\cite{nematicelectro} and~\cite{approx} for details).
In the article at hand, we additionally show that a dissipative solution enjoying additional regularity is in fact a strong solution. Thus, it can be argued that the presented generalized solution concepts are qualitatively the same in terms of \textit{existence}, \textit{stability}, \textit{weak-strong uniqueness}, and \textit{regularity implying uniqueness}.  While dissipative solutions do not fulfill the equation, they have less degrees of freedom compared to measure-valued solutions. 
Since dissipative solutions appear natural as a singular limit, it may be worth considering the singular sharp interface limit of $\varepsilon\ra0$ in the relative energy inequality~(see~\eqref{relenenergy} below). 
Especially since there  exists a formulation of a relative energy for the sharp interface case (see~\cite{julian1}) and even a convergence proof for vanishing interface thickness using this technique (see~\cite{julian2}).

%
 
In this article, we want to consider the system~\eqref{eq} as a \textit{prototype system} for a more general GENERIC system, \textit{i.e.}, it is a thermodynamical consistent system coupling the incompressible Navier--Stokes system to an energy balance and an additional Gradient-Flow-like equation for the evolution of an internal variable. 
In the sequel, the existence of weak and measure-valued solutions is proven, as well as the weak-strong (or rather measure-valued-strong) uniqueness of theses solutions. Weak solutions only emerge  under additional (possibly unnatural) regularizing terms appearing in heat-capacity or heat conduction. 
The measure-valued formulation consists of an entropy production rate in a distributional sense (see~\eqref{weakentro} below), an energy inequality (see~\eqref{energyin} below), and an entropy inequality (see~\eqref{entropy} below),  in a nowadays standard way (compare for instance to~\cite{singular,larosc}) as well as the weak formulation of the Cahn--Hilliard equation (see~\eqref{phaseeq} and~\eqref{mueq} below). 
The Navier--Stokes-like equation is fulfilled in a measure-valued sense~(see~\eqref{weakmomentum} below) unless additional regularity is available~(compare to Theorem~\ref{thm:existence} below). 
Additionally, the existence of  dissipative solutions is shown, which inherit the weak-strong uniqueness property by construction. 
In comparison to the measure-valued formulation, in the definition of  dissipative solutions the momentum balance as well as the energy and the entropy inequality are replaced by a relative energy inequality~(see~\eqref{relenenergy} below).  
This deprives us from establishing strong convergence of an approximate sequence of the velocity field $\f u$, but the formulation is weakly sequential stable with respect to the weak topology for the velocity fields. 
%
%
The dissipative formulation has the advantage that it does not rely on a measure in the \textit{elastic stress tensor}. This makes the formulation more attainable for structure inheriting discretizations (compare~\cite{nematicelectro,approx}). The formulation still relies on a measure-valued formulation of the entropy balance, but this is mainly to have some control of the time derivative of the entropy in order to deduce some strong convergence (compare Remark~\ref{rem:form}). 

Up to the best knowledge of the author, this is the first time that these different solution concepts are considered for such a thermodynamical consistent systems.

\textbf{The paper is organized as follows:} In the remaining part of this section it is shown formally that the considered model fits into the GENERIC modeling concept. In Section~\ref{sec:sec2}, notation, assumptions, and main results are collected. Additionally, some useful lemmata are provided. Section~\ref{sec:ex}, executes the existence proof for measure-valued solutions, which also lies at the core of the existence proof for the dissipative solution concept. 
Finally, in Section~\ref{sec:rel}, the relative energy inequality is proven and thus, the weak-strong uniqueness for measure-valued solutions and the existence of dissipative solutions.

\subsection{Modeling\label{sec:mod}}

In this subsection, we comment on the modeling of the considered system~\eqref{eq}. The calculations presented in this section are purely formal.
The system~\eqref{eq} may be modeled via Fremond's approach see \cite{fremond} and \cite{giulio}. 
It may also be derived by following the GENERIC modeling approach. 
GENERIC stands for General Equation for Non-Equilibrium Reversible-Irreversible Coupling  and was proposed by M. Grmela and H.C Öttinger see~\cite{generic}
It states that the evolution of a thermodynamical consistent system may be expressed on a state space $\mathcal{N}$ via 
\begin{align}
\t \f q = \mathcal{J}(\f q) D\mathcal{E}(\f q) + \mathcal{K}(\f q)D\mathcal{S}(\f q)\,,\label{generic}
\end{align}
where $\mathcal{E}: \mathcal{N}\ra \R $ and $\mathcal{S}: \mathcal{N} \ra \R $ are the energy and entropy of the system, respectively, and $\mathcal{J}
$ an anti-symmetric Poisson structure $(\mathcal{J}(\f q) = - \mathcal{J}(\f q)^*)$ fulfilling the Jacobi-identity and $\mathcal{K}
$ the symmetric dissipative structure $ (\mathcal{K}(\f q) = \mathcal{K}(\f q)^*)$, which is positive semi-definite, \textit{i.e.}, $ \langle \f \xi , \mathcal{K}(\f q) \f \xi \rangle \geq 0$
on a underlying manifold with the non-interaction conditions 
\begin{align}
\mathcal{J}(\f q ) D\mathcal{S}(\f q) = 0 = \mathcal{K}(\f q ) D \mathcal{E}(\f q )\,.\label{noninter}
\end{align}
%
The free energy density $\psi$ of the system is given by
\begin{align}
\psi(\theta,\varphi ) = \frac{\varepsilon}{2}| \nabla \varphi |^2 + \frac{1}{\varepsilon} F( \varphi ) + f_\delta( \theta ) - \theta \varphi  
\quad\text{such that}\quad
 \Psi ( \theta , \varphi ) = \int_\Omega \psi(\theta, \varphi ) \de \f x \,, 
\end{align}
where $f_\delta $ represents the purely caloric heat part of the free energy and is given by
\begin{align}
f_\delta ( \theta) : = \begin{cases}
- \theta ( \log \theta - 1) & \delta = 0\\
- \frac{1}{\delta(\delta+1)} \theta^{\delta+1} & \delta>0\,
\end{cases}\,.\label{fdel}
\end{align}
For convenience, we define $ Q_\delta(\theta) = f_\delta(\theta)- \theta f_\delta'(\theta)$. The specific heat is given by $ c_V(\theta) = Q'_\delta  (\theta) $. 
We may derive the entropy density by
\begin{align}
s(\theta , \varphi ) = - \frac{\delta \Psi}{\delta \theta} = - f'_\delta (\theta) + \varphi\label{Na:s} \quad \text{such that} \quad \mathcal{S}(\f q) =  \int_\Omega s ( \f x) \de \f x \,.
\end{align}
Setting $\mathcal{N}:= H^1(\R)\times H^1(\R) \times \f V$ (see the subsequent section for the definition of $\f V$), the  energy may be expressed as the sum of internal and kinetic energy
\begin{align*}
\mathcal{E}(\f q  ) =\int_\Omega Q_\delta(\theta ) + \frac{1}{2}| \nabla \varphi |^2 + F(\varphi)+ \frac{1}{2}| \f u |^2  \de \f x = \Psi (\theta , \varphi) + \int_\Omega \theta s  \de \f x + 
\frac{1}{2}\int_\Omega  | \f u |^2  \de \f x 
\,,
\end{align*}
where $\f q = (s , \varphi , \f u)$ and $\theta $ has to be interpreted as $ \theta (s, \varphi) = (f_\delta ' )^{-1}(\varphi - s )$. 
The form $\mathcal{J}$ 
may be expressed via
\begin{align*}
 \langle \eta , \mathcal{J} ( \f q ) \zeta \rangle = \int_\Omega \eta \cdot \begin{pmatrix}
 0&0&-\nabla s \\ 0&0& - \nabla \varphi \\ \nabla s & \nabla \varphi & (\f u \cdot \nabla ) - \di (\f u \cdot) 
\end{pmatrix}\zeta \de \f x \,.
\end{align*}
The skew-symmetry of $\mathcal{J}$ follows from the fact that $\f u$ is divergence free and an integration-by-parts.  
The symmetric form $ \mathcal{K}$ has to be chosen to guarantee 
$$ \mathcal{K}(\f q) \begin{pmatrix}
1\\0\\0
\end{pmatrix} = \begin{pmatrix}
\kappa (\theta ) | \nabla \log \theta |^2 + \frac{\nu(\theta)| (\nabla \f u)_{\sym}|^2}{\theta} + \frac{| \nabla \mu|^2}{\theta} + \di ( \kappa(\theta) \nabla \ln \theta)  \\
\Delta \mu 
\\
\di ( \nu(\theta ) ( \nabla \f u )_{\sym} ) ) 
\end{pmatrix} 
\,$$
and the additional conditions on the dissipative form, \textit{i.e.}, symmetry and non-interaction condition~\eqref{noninter}. 
This implies the following form of $\mathcal{K}$, \textit{i.e.},
\begin{align*}
\left \langle \eta , \mathcal{K}(\f q ) \zeta \right \rangle ={}
\int_\Omega&  \eta_1 \zeta_1 \left (\kappa (\theta ) | \nabla \log \theta |^2 + \frac{\nu(\theta)| (\nabla \f u)_{\sym}|^2}{\theta} + \frac{| \nabla \mu|^2}{\theta} \right ) 
\\&  
+ \kappa(\theta)\left ( \nabla \eta_1 \cdot \nabla \zeta_1- \nabla \eta_1 \cdot\nabla  \log \theta \zeta_1 - \eta_1 \nabla \log \theta \cdot \nabla \zeta _1\right )
\\
& - \nabla \eta_1 \cdot \nabla \mu \zeta_2 - \eta_1 \nabla \mu\cdot \nabla \zeta_2 + \nabla \eta_1 \mu \cdot \nabla  \zeta _2
\\& - \nu(\theta)\left ( \nabla \eta_1 \cdot  ( \nabla \f u)_{\sym} \cdot \f \zeta_3 + \eta_1  ( \nabla \f u)_{\sym} : (\nabla \f \zeta_3)_{\sym} - \nabla \eta_1 \otimes  \f u  :( \nabla  \f \zeta _3)_{\sym}\right )
\\& - \nabla \eta_2 \cdot \nabla \mu \zeta_1 - \eta_2 \nabla \mu\cdot \nabla \zeta_1 + \nabla \eta_2 \mu \cdot \nabla  \zeta _1
\\
& - \nabla \eta_2 \cdot \nabla \theta  \zeta_2 + \eta_2 \nabla \theta \cdot \nabla \zeta_2 + \nabla \eta_2 \theta  \cdot \nabla  \zeta_2 
\\& -\nu(\theta)\left (( \nabla\f  \eta_3)_{\sym} :  ( \nabla \f u)_{\sym}   \zeta_1 +   ( \nabla \f u)_{\sym} : \f \eta_3 \otimes \nabla  \zeta_1  - (\nabla \f \eta_3)_{\sym}  :  (\f u  \otimes \nabla   \zeta _1)\right )
\\& -\nu(\theta)  \left ( ( \nabla\f  \eta_3)_{\sym} :  ( \nabla \theta \otimes  \f \zeta_3)-\left   ( \nabla \theta \otimes  \f \eta_3+\theta  (\nabla \f \eta_3)_{\sym} \right ) : ( \nabla  \f \zeta_3)_{\sym} 
\right ) 
\de \f x \,,
\end{align*}
%
%
where the occurrences of $\theta$ again have to be expressed via $s$ and $\varphi$ using~\eqref{Na:s}. 

From $D\mathcal{S} = (1,0,0)^T$ and $D\mathcal{E} = (\theta , \mu , \f u^T )^T$ we observe the system\eqref{eq} by the standard GENERIC approach~\eqref{generic}. 
\section{Preliminaries and main results\label{sec:sec2}}
In this section, the assumptions and notations are given, as well as the main results. Additionally some preliminary lemmata are provided, which may be interesting in their own right. 
\subsection{Assumptions and notation\label{sec:not}}
We introduce some notation. Let $\Omega \subset \R^d$ be a bounded sufficiently smooth domain and $d \geq2$. As usual, $ \R_+ :=\{ r \in \R , r \geq 0\}$. 
We denote by $\pmb{\mathcal{V}}:=\{ \f v \in \mathcal{C}_{c}^\infty(\Omega;\R^d)| \di \f v =0\}$ 
the space of smooth solenoi\-dal functions with compact support. By ${\f H} $, ${\f V}$, and $\f W^{k,p}_{0,\sigma}(\Omega)$
we denote the closure of $\pmb{\mathcal{V}}$ with respect to the norm of $\f L^2(\Omega) $,   $ \f H^1( \Omega) $, and $\f W^{k,p}(\Omega)$,
respectively (for $k \in \N$, $p \in (1,\infty)$).
Note that ${\f H}$ can be characterized by ${\f H} = \{ \f v \in {\f L}^2 | \di \f v =0 \text{ in }\Omega\, , \f n \cdot \f v = 0 \text{ on } \partial \Omega \} $, where the first condition has to be understood in the distributional sense and the second condition in the sense of the trace in ${H}^{-1/2}(\partial \Omega )$. 
The dual space of a Banach space ${\mathbb X}$ is always denoted by ${\mathbb X}^*$ and is equipped with the standard norm; the duality pairing is denoted by $\langle\cdot, \cdot \rangle$.
For $Q\subset \R^d$, the Radon measures are denoted by $\mathcal{M}(Q)$. We recall that the Radon measures equipped with the total variation are a Banach space and  for compact sets $Q$ , it can be characterized by~$\mathcal{M}(Q) = ( \C(Q))^*$ (see~\cite[Theorem~4.10.1]{edwards}). 
The integration of a function $f\in \C(Q)$ with respect to a measure $\mu\in \mathcal{M}(Q)$ is denoted by $ \int_Qf(\f h ) \mu(\de \f h)\,.$ In case of the Lebesgue measure we just write 
$ \int_Qf(\f h ) \de \f h\,.$

By $ \R^{d\times d }_{\sym,+}$, we denote symmetric positive semi-definite matrices. For a matrix $\f A\in \R^{d \times d} $, we denote the negative symmetric part by $  ( \f A )_{\sym,-}$, which is given by 
$  ( \f A )_{\sym,-} := \inf_ {\{ \f a \in \R^d | \f a | =1 \}, }  \f a \cdot \f A \f a $.
Associated to $f_\delta$ defined in~\eqref{fdel}, we define the thermal energy by $ Q_\delta(\theta) = f_\delta(\theta)- \theta f_\delta'(\theta)$. The specific heat is given by $ c_V(\theta) = Q' (\theta) $. 
 To abbreviate, we define $\Lambda_\delta (\theta |\tet) = f_\delta(\theta)-f_\delta(\tet) - f'_\delta(\theta) (\theta - \tet) $ .  
Note that $f_\delta $ is concave such that $\Lambda_\delta $ is positive.
By $c$, we denote a generic constant, which may changes from line to line. 

We will need a number of assumptions on $F$, $\kappa$, and $f_\delta$, namely
\begin{hypothesis}
\label{hypo}
(A)~~We let $ F \in \C^2(\R, \R) \cap \C^{2,1}_{\text{loc}}( \R,\R)$. 

\smallskip
\noindent%
(B)~~We assume $F$ to be {\it $\lambda$-convex}, \textit{i.e.}, convex up to a quadratic 
perturbation. Namely there exists a constant $\lambda>0$ such that 
$ F''(y) \geq - \lambda$ for all $ y \in \R$.
We can then define a convex modification of $F$, subsequently named $G$, as 
\begin{align}
  G(y ) = F( y) + \lambda y^2 \, \quad y \in \R\label{lambdacon}.
\end{align}
By construction, $G$ is ``strongly convex'', \textit{i.e.}, $G''(y)\ge \lambda >0$
for all $ y \in \R$. Moreover, it is not restrictive to assume $G$ to 
be nonnegative and so normalized that $G'(0)=0$. 

\smallskip
\noindent%
(C)~~Next, we assume a  coercivity assumption at $\infty$,
namely 
\begin{equation}\label{Fcoerc}
  \liminf_{|y|\ra \infty} \frac{F(y)}{|y|}   = +\infty.
\end{equation} 
As a consequence of \eqref{Fcoerc}, we can first observe that 
$F(y) \geq -c $ for  some constant $c>0$ and every $y\in \R$.
Moreover, it is easy to verify that
the physical energy controls the $H^1$-norm of $\fhi$ {from above}. Namely, 
there exist $\gamma > 0$ and $c\ge 0$ such that
\begin{equation}\label{Ecoerc}
  \frac12 \| \nabla \fhi \|_{L^2(\Omega)}^2
   + \int_\Omega F(\fhi) \, \de \f x \ge \gamma \| \fhi \|_{H^1(\Omega)} - c,
\end{equation} 
for every $\fhi \in H^1(\Omega)$ such that $F(\fhi) \in L^1(\Omega)$.

\smallskip
\noindent%
(D)~~Finally, a growth condition is assumed to hold, \textit{i.e.}, 
there exists a constant $c>0$ such that 
\begin{equation}\label{growthF}
  | F' (y) |\log (e + | F'(y) |) \leq c( 1 + |F(y)| ) 
   \quad\text{for all }y \in \R.
\end{equation} 
Possibly modifying the value of $c$ one can see that the
analogue of \eqref{growthF} holds also for the convex 
modification~$G$, \textit{i.e.}, we have
\begin{equation}\label{growthG}
  | G' (y) |\log (e + | G'(y) |) \leq c( 1 + G(y) ) 
   \quad\text{for all }y \in \R.
\end{equation} 
To check that \eqref{growthF} implies \eqref{growthG},
a number of straightforward but somehow technical
computations would be required. We leave them to the reader 
because no real difficulty is involved. 

For simplicity, we assume for $k_0 >0 $ and $k_1 \geq 0$ that
\begin{align}
\kappa (\theta) = \kappa_0 + \kappa_1 \theta^\beta \quad \text{with } \beta \in [0,2] \quad \text{and}\quad 
f_\delta ( \theta) : = \begin{cases}
- \theta ( \log \theta - 1) & \delta = 0\\
- \frac{1}{\delta(\delta+1)} \theta^{\delta+1} & \delta>0\,
\end{cases}
\text{ for } \delta \in \left [0,1\right   )
\,.
\end{align}
For $\delta = \beta /2$, we may also allow $\kappa_0 = 0$. 
Additionally, we assume $ 0 < \underline{\nu } \leq \nu(s) \leq c $ for all $s \in \R_+$. 
Moreover, 
we define $\hat{\kappa}(r)= \int_1^r \kappa (s)/ s \de s $. 

\end{hypothesis}

We define the set $\mathbb{X}$ via: $ (\f u, \theta, \varphi
 ) \in \mathbb{X}$ if 
\begin{align*}
\f u & \in L^\infty(0,T; L^2(\Omega) ) \cap L^2(0,T; \f V) \,,\\
Q_{\delta }(\theta)  & \in L^\infty( 0,T; L^1(\Omega)) \quad \text{with } \theta(x ,t) > 0 \text{ a.e. in }\Omega \times (0,T) \text{ and } \theta \log \theta \in L^1(\Omega \times (0,T))  \,,\\
f'_\delta (\theta)  & \in L^\infty(0,T;L^1(\Omega)) \quad \text{with } \hat{\kappa}(\theta) \in  L^2(0,T; H^1(\Omega))\quad \text{and }\t f'_\delta (\theta) \in \mathcal{M}([0,T];(W^{1,p}(\Omega))^*) \,,\\
\varphi &\in L^\infty (0,T; H^1(\Omega)) \quad \text{with }\t \varphi  \in L^1(0,T;(W^{1,p}(\Omega))^*) 
\,.
\end{align*}

We define the set $\mathbb{Y}$ of  regular solutions via: $(\tu,\tet,\tp
)\in \mathbb{Y}$ if $(\tu,\tet,\tp
)\in \mathbb{X}$ and additionally
\begin{align*}
\tu \in{}& L^1(0,T;\C^1(\ov\Omega)) \cap \mathcal{C}^1([0,T];\f V)\,,\\
\tet \in {}& L^1(0,T;W^{2,\infty}(\Omega)) \cap L^2(0,T;W^{1,\infty}(\Omega))\cap L^\infty(\Omega \times (0,T)) \\
&\text{ with } \tet \geq \underline{\tet}>0 \text{ a.e. in }\Omega \times (0,T) \quad \text{and } \f n \cdot \kappa(\theta ) \nabla \theta = 0\text{ on }\partial \Omega \times (0,T)  \,,\\
\tp \in {}& W^{1,1}(0,T;L^\infty(\Omega)) \cap L^1(0,T;W^{3,p }(\Omega))\quad \text{for } p>d\quad \text{with }\f n \cdot \nabla \tp = 0 \text{ on }\partial \Omega \times (0,T) 
\,.
\end{align*}

\subsection{Main results}
In the following, we collect the main results of the paper. 
\begin{definition}\label{def:weak}
An element  $ (\f u, \theta, \varphi
) \in \mathbb{X}$, a chemical potential $\mu  \in L^2(0,T;W^{1,1}(\Omega)) 
$,  and a defect measure $m \in L^\infty(0,T;\mathcal{M}(\overline\Omega; \R^{d\times d } _{\text{sym},+} ))$ is called a measure-valued solution to~\eqref{eq}, if 
\begin{align*}
\frac{|(\nabla \f u)_{\sym}| }{\sqrt{\theta}}\,,    \frac{|\nabla \mu |}{\sqrt{\theta}} \in L^2(0,T;L^2(\Omega)) \,,
\end{align*}
   the incompressibility condition $ \di \f u =0$ is fulfilled a.e.~in $\Omega \times (0,T)$, the weak momentum balance 
 \begin{multline}
-  \int_\Omega \f u \cdot \f \xi \de \f x \Big |_0^t + \int_0^t \int_\Omega \f u \t \f \xi + ( \f u \otimes \f u ) : \nabla \f \xi 
\de \f x \de t\\ = \int_0^t \int_\Omega \nu (\theta ) ( \nabla \f u )_{\sym} : (\nabla \f \xi )_{\sym} -\varepsilon  (\nabla \varphi \otimes \nabla \varphi ) : (\nabla \f \xi )_{\sym} \de \f x - \int_{\overline\Omega}( \nabla \f \xi)_{\sym} : m(\de \f x) \de t \label{weakmomentum}
 \end{multline}
 for all $\f \xi \in \C_c^\infty( \Omega \times  [0,T)) $ with $ \di \f \xi = 0$ in $\Omega \times (0,T)$ and for a.e.~$t\in 0,T)$, the entropy inequality 
\begin{multline}
  \int_{\Omega} \vartheta  (f'_\delta  (  \theta) - \varphi) \de \f x  \Big|_0^t + \int_0^t \int_{\Omega} \vartheta  \left ( \kappa (\theta)  | \nabla  \log \theta|^2+ \frac{\nu(\theta) | (\nabla \f u_{\sym}|  ^2}{\theta }+ \frac{|\nabla  \mu |^2 }{\theta}  \right ) \de \f x  \de s\\
  \leq \int_0^t \int_\Omega\left (  \kappa (\theta)    \nabla \log \theta \cdot \nabla \vartheta + ( \vartheta _t + (\f u \cdot \nabla ) \vartheta) ( f'_\delta ( \theta) - \varphi )\right ) \de \f x  \de s \label{entropy}
\end{multline}
for all $ \vartheta \in L^\infty(0,T; L^\infty( \Omega)) \cap W^{1,1}(0,T; L^\infty(\Omega))\cap L^2 (0,T; W^{1,d}(\Omega))$ and for a.e.~$t\in(0,T)$, and the weak formulation of the Cahn--Hilliard part
\begin{align}
\int_0^t\left \langle \t \varphi , \phi \right \rangle -  \int_\Omega \left ( \varphi  \f u \cdot \nabla \right )  \phi - \nabla \mu  \cdot \nabla \phi  \de \f x \de s  =  0 \label{phaseeq}
\end{align}
with
\begin{align}
 \mu 
 + \varepsilon
 \Delta 
  \varphi 
  - \frac1\varepsilon F'(\varphi )
  + \theta 
   =0\,, \quad \text{a.e.~in }\Omega \times (0,T) \label{mueq} 
\end{align}
 for all $ \phi \in  L^2(0,T; W^{1,d}(\Omega)) \cap L^\infty(0,T; W^{1,\infty} (\Omega)) $ and with $ \f n \cdot \nabla \phi = 0 $ a.e.~on $\partial \Omega \times (0,T)$. 
 Note that the trace $\f n \cdot \nabla \varphi$ is well defined in $L^1(\partial \Omega)$ (see~\cite[Prop.~3.80]{demengel}).
Additionally, the energy inequality holds:
\begin{multline}
\frac{1}{2}\int_{\Omega } \left (|\f u (t)|^2 +\varepsilon | \nabla \varphi(t)|^2 + \frac{1}{\varepsilon} F( \varphi(t)) + Q_\delta (\theta(t)) 
 \right )\de \f x +\frac{1}{2}\langle m , 1 \rangle   \\   \leq \frac{1}{2}\int_{\Omega } \left (| \f u_0|^2+ \varepsilon	| \nabla \varphi_0|^2 + \frac{1}{\varepsilon}  F( \varphi_0)  + Q_\delta (\theta(t))
 \right ) \de \f x  \,\label{energyin}
\end{multline}
and a weak form of the entropy balance, \textit{i.e.} there exists a measure $ \eta \in \mathcal{M}( \ov \Omega \times [0,T])$ such that 
\begin{subequations}
\label{weakentro}
\begin{multline}
\left \langle \t ( f'_\delta(\theta)  - \varphi ) , \chi \right \rangle _{\mathcal{M}([0,T]; (\f W^{1,p}(\Omega))^* ), \C ([0,T];\f W^{1,p}(\Omega)) }  + \left \langle \eta , \chi\right \rangle_{\mathcal{M}(\ov \Omega \times [ 0,T] ), \C (\ov \Omega \times [ 0,T]) } \\ = \int_0^T \int_\Omega\left (  \kappa(\theta ) \nabla \theta + \f u f_\delta '(\theta)\right ) \cdot \nabla \chi \de \f x \de t \,
\end{multline}
for all $\chi \in \C ([0,T];\f W^{1,p}(\Omega)$ for $p>d$, where the measure $\eta $ may be bounded from below by
\begin{align}
 \left \langle \eta , \chi\right \rangle_{\mathcal{M}(\ov \Omega \times [ 0,T] ), \C (\ov \Omega \times [ 0,T]) }  \geq \int_0^T\int_{\Omega} \chi  \left ( \kappa (\theta)  | \nabla  \log \theta|^2+ \frac{\nu(\theta) | \nabla \f u|  ^2}{\theta }+ \frac{|\nabla  \mu |^2 }{\theta}  \right ) \de \f x  \de t
\end{align}
for all $\chi \in \C (\ov \Omega \times [0,T])$ with $\chi \geq 0$ in $\Omega \times (0,T)$. 
\end{subequations}

\end{definition}
\begin{remark}
As a defect measure we understand a measure, which has Lebesgue-part zero, \textit{i.e.}, is concentrated on a set of Lebesgue measure zero. Sometimes such a measure is also referred to as concentration measure, since it captures concentrations of the approximating sequence. This has to be understood in contrast to the usual Young measure, or oscillation measure that captures oscillations of an approximate sequence. 
We can exclude concentrations here, due to point-wise  a.e.~convergence of the approximate sequences (compare to~\cite{meas}). 
\end{remark}

\begin{theorem}
\label{thm:existence}
Let $ \Omega\subset \R^d $ be sufficiency smooth 
 and Hypothesis~\ref{hypo} be fulfilled with $2\delta \leq  \beta$. To every $Q_\delta( \theta_0 ) \in L^1(\Omega) $ with $ f'_\delta ( \theta_0)\in L^1(\Omega)$ and $ \theta _0 (\f x ) >0 $ a.e.~in $\Omega$ as well as $ \varphi_0 \in H^1(\Omega)$ with $F(\varphi_0)\in L^1(\Omega)$ and $\f u_0\in \f H$, there exists at least one measure-valued solution according to Definition~\ref{def:weak}.

If the requirement
\begin{align}
\delta > \frac{d}{2}\left ( \frac{d}{2}-\beta\right ) \label{addreg}
\end{align}
is fulfilled, the defect measure $m$ vanishes, \textit{i.e.}, $ m=0$. 
\end{theorem} 
  
 \begin{theorem}\label{thm:main}
 Let Hypothesis~\ref{hypo} hold true. Let $(\f u,  \theta , \varphi  )\in \mathbb{X}$ be a weak solution according to~\ref{def:weak} and $( \tu, \tet ,\tp )\in \mathbb{Y}$ a strong solution  emanating from the same initial data. 
Let $\beta\in [4\delta,2-2\delta ]$.

Then both solution coincide,\textit{ i.e.}, $ \f u = \tu$, $ \theta \equiv \tet $, and $\varphi \equiv \tp $. 

For $\delta > 0 $, the weak-strong uniqueness result also holds for $\beta \leq 4\delta$ under the additional assumption that the solution $\theta$ is bounded pointwise from below, \textit{i.e.}, $ \theta \geq  \underline{\theta}>0$ a.e.~in $\Omega \times (0,T)$. 
 \end{theorem}
Theorem~\ref{thm:main} is a direct consequence of Proposition~\ref{prop:rel}. 

 \begin{remark}
 By the presented technique, the weak-strong uniqueness result for a solution with $\kappa_0>0$, \textit{i.e.}, with the standard part of Fourier's law, only holds, if $\delta=0$. In a sense, the resulting dissipative logarithmic terms in the entropy balance~\eqref{entropy} (the terms multiplied by $\kappa(\theta)$), can only be estimated by associated logarithmic terms in the entropic part ($f'_\delta$ for $\delta=0$ in~\eqref{entropy}). 
 
In the case that the solution $\theta$ is bounded pointwise from below, \textit{i.e.,} $ \theta \geq  \underline{\theta}$ a.e.~in $\Omega \times (0,T)$, this restriction does not occur since in this case, there exists a $C>0$ such that
\begin{align*}
| f'_{\delta_1} ( \theta ) - f'_{\delta _1} ( \tet) | \leq C | f'_{\delta_2} ( \theta ) - f'_{\delta _2} ( \tet) |  \quad \text{for }\delta_1 \leq \delta _2 \quad\text{ and for } \theta,\tet \geq \min\{ \underline{\theta}, \underline{\tet}\}
\,.
\end{align*}
But it it seems to be out of reach to show that such a lower bound holds for the considered cases.

 \end{remark}
 
 The proof of Theorem~\ref{thm:main} relies on the fact that a solution according to Definition~\ref{def:weak} fulfills a so-called relative energy inequality. 
 In the following, we restrict ourselves to the case $\delta=0$, since this is the important case and the relative energy inequality holds without an additionally assumed lower bound as in Theorem~\ref{thm:main}. 
%
%
%
%
In the case of a convex energy functional, this idea goes back to {Dafermos}~\cite{dafermos} in the context of thermodynamical systems.
For a strongly convex {G\^{a}teaux} differentiable  energy functional $\mathcal{E}:\mathbb{X}  \ra \R$  the relative entropy of two solutions~$u$ and $\tilde{u}$ is given by (see \cite[Sec.~5.3]{dafermos2})
\begin{align}
\mathcal{R}(u|\tilde{u})= \mathcal{E} (u) - \mathcal{E}(\tilde{u}) - \langle\mathcal{E}'(\tilde{u}), u -\tilde{u}\rangle  > 0 \text { for } u \neq \tilde{u} \,.\label{relencon}
\end{align}
The  strong convexity of $\mathcal{E}$ guarantees that $\mathcal{R}$ is positive as long as $ u$ and $\tilde{u}$ do not coincide.
Let us consider the nonlinearity $F$. A function fulfilling Hypothesis~\ref{hypo} is called $\lambda$-convex, \textit{i.e.}, the function is convex up to an additive shift by the identity. 
According to~\eqref{lambdacon}, the  convex modification of $F$ is called $G$. 
%
The relative energy for system~\eqref{eq} is defined via
\begin{subequations}
\label{relen}
\begin{align}
\mathcal{R}(\f q  | \tilde {\f q}): ={}& \int_{\Omega} \left (  \frac{\varepsilon}{2}| \nabla \varphi - \nabla \tp |^2 +\frac{1}{\varepsilon}\left ( G ( \varphi ) - G ( \tp ) - G'(\tp)( \varphi - \tp )   \right ) \right ) \de \f x\label{relen1} \\
& +\frac{1}{2}\| \f u - \tu \|_{L^2(\Omega)}^2+ \int_\Omega \Lambda_{\delta} \left ( \theta | \tet\right )
   \de \f x \label{relen2}
+ \frac{M}{2} \| \varphi - \tp\|_{(W^{1,\infty}(\Omega))^*}^2 
\,,
\end{align}
\end{subequations}
where we defined $ \f q =( \f u, \theta , \varphi )$ and $ \tilde {\f q} = (\tu,  \tet,\tp)$. 
Due to the convexity of $G$, we may conclude by choosing~$\eta = G $, $ u = \varphi$, and $\tilde{u}= \tp$ in~\eqref{relencon} that the line~\eqref{relen1} is nonnegative.
%
Due to definition~\eqref{lambdacon},   we find for the convex function $G$ that 
\begin{align*}
G ( \varphi ) - G ( \tp ) - G'(\tp)( \varphi - \tp )  = {}& F ( \varphi ) - F ( \tp ) - F'(\tp)( \varphi - \tp )  \\
&+ {\lambda}( | \varphi|^2 - | \tp |^2 - 2\tp ( \varphi - \tp )) \\
={}& F ( \varphi ) + F ( \tp ) - (F'(\tp)( \varphi - \tp )  +2 F(\tp)) + {\lambda} | \varphi - \tp |^2 \,.
\end{align*}
To handle the last term, which is due to the non-convexity of $F$, we add another rather weak norm to the relative energy. 
For $M$ big enough, we find by an interpolation inequality
\begin{align*}
 \lambda \| \varphi - \tp\|_{L^2(\Omega)} ^2 \leq \frac{1}{4} \| \nabla \varphi - \nabla \tp\|_{L^2(\Omega)}^2 + \frac{M}{2} \| \varphi - \tp\|_{(W^{1,\infty}(\Omega))^*}^2 \,. 
\end{align*}
Note that this only holds since $ \int_\Omega \varphi - \tp \de \f x = 0$ and the $H^1$-semi norm is equivalent to the full $H^1$-norm for $\varphi - \tp$. 
Indeed, considering the Poisson equation $ - \Delta \Phi = \varphi - \tp $ in $\Omega$ with $ \f n \cdot \nabla \Phi =0 $ on $ \partial \Omega$ and $\int_\Omega \Phi \de \f x = 0$, we find
\begin{align*}
\| \varphi - \tp \| _{L^2(\Omega)}^2 ={}& \left ( \nabla \varphi - \nabla \tp , \nabla \Phi \right ) \leq \| \nabla \varphi - \nabla \tp \|_{L^2(\Omega)} \| \nabla \Phi \|_{L^2(\Omega)} 
\\ \leq{}& c\| \nabla \varphi - \nabla \tp \|_{L^2(\Omega)} \|  \Phi \|_{H^2(\Omega)}^{d/(2+d)} \| \nabla \Phi \|_{L^1(\Omega)} ^{2/(2+d)} 
\\ \leq{}& \frac{1}{2} \| \varphi -\tp \|_{L^2(\Omega)}^2 + \frac{1}{2} \| \nabla \varphi -\nabla \tp \|_{L^2(\Omega)}^2 + c \| \nabla \Phi \|_{L^1(\Omega)}^2\,.
\end{align*}
The first term can be absorbed on the left-hand side and by Hahn--Banach's theorem the last term can be identified  via
$$\| \nabla \Phi \|_{L^1(\Omega)}=\sup_{\f a \in L^\infty, \| \f a \|_{L^\infty(\Omega)}=1} \left \langle \nabla \Phi , \f a \right \rangle \,.$$
 
Additionally, we define the relative dissipation by 
\begin{align}
\begin{split}
\mathcal{W} (\f q  |\tilde {\f q} ) : ={}&   \int_0^t { \kappa_0} \int_{\Omega}  \tet | \nabla \log \theta - \nabla \log \tet |^2 \de \f x 
+ \int_\Omega   \nu (\theta)\frac{1}{2} \left |\sqrt\frac{\tet }{{\theta}} (\nabla\f u)_{\sym} 
    - \sqrt\frac{\theta }{{\tet}}( \nabla \tu)_{\sym} \right |^2
      \de \f x  \de s
\\+&\int_0^t \int_\Omega \frac{\kappa_1}{\beta^2 }  \tet | \nabla {\theta}^{\beta/2}  - \nabla {\tet}^{\beta/2} |^2 +    \left | \sqrt\frac{\tet }{\theta } \nabla \mu - \sqrt\frac{\theta}{{\tet}}\nabla \tilde{\mu}\right |^2  \de \f x  \de s 
\,.
\end{split}\label{W}
\end{align} 
 the regularity measure by 
 \begin{align*}
 \mathcal{K} (\tilde {\f q} )  := {}&c \left ( \| \t \tp+ ( \tu \cdot \nabla ) \tp  \|_{L^\infty(\Omega)}  + \left \|\varepsilon \Delta \tp -\frac{1}{\varepsilon} F'(\tp) \right \|_{W^{1,d}(\Omega)} + \left \| \left (\t + (\tu \cdot \nabla )\right ) \ln \tet \right \|_{L^\infty(\Omega)} \right ) \\& +c\left ( \left \|
  \nabla \tilde{\mu }
  \right \|_{L^\infty(\Omega)}^2 + \left \| 
  (\nabla \tu)_{\sym} 
  \right  \|_{L^\infty(\Omega)} ^2+ \kappa_0\left  \| 
 {\Delta \tet}
 \right \| _{L^\infty(\Omega)} \right )\\& + c\kappa_1 \left ( \| \nabla \tet^{\beta/2} \|_{L^\infty(\Omega)}^2 + \| \Delta \tet^{\beta/2} \|_{L^\infty(\Omega)}\right ) + c \| \nabla \tilde{\mu }\|_{L^2(\Omega)}\,
 \end{align*}
 and the solution operator by
 \begin{align}
 \left \langle \mathcal{A}(\tilde{\f q}  )  , \bullet \right \rangle :={}\left \langle \begin{pmatrix}
 \t\tu+(\tu\cdot\nabla)\tu - \di ( \nu(\tet)\nabla \tu ) + \varepsilon \di ( \nabla \tp\otimes \nabla \tp) 
 \\ 
 (\t  +( \tu \cdot\nabla )) (- \ln \tet - \tp ) +  \frac{1}{\tet}\di ( \kappa (\tet)\nabla  \tet) 
 + \nu(\tet) \frac{| (\nabla \tu )_{\sym}|^2 } {\tet}+ \frac{ | \nabla \tilde{\mu}|^2 }{\tet}
 \\
 \t \tp + ( \tu \cdot \nabla) \tp - \Delta \tilde{\mu} 
 \end{pmatrix}
,
\bullet \right \rangle\,,\label{operatorA}
 \end{align}
 where the chemical potential $\tilde{\mu}$ is given by $ \tilde{\mu} =-\varepsilon \Delta \tp +\frac{1}{\varepsilon} F'(\tp) - \tet $ with $ \f n \cdot \nabla \varphi = 0$.

With this notation at hand, the relative energy inequality is given by  
\begin{multline}
\mathcal{R}(\f q  |\tilde{\f q} ) (t) + \int_0^t \mathcal{W}(\f q  |\tilde{\f q} )  e^{\int_s^t  \mathcal{K} (\tilde{\f q}) \de \tau } \de s 
\leq 
\mathcal{R}(\f q  |\tilde{\f q} ) (0) e^{\int_0^t  \mathcal{K} (\tilde{\f q}) \de s }
\\+ \int_0^t \left (  \left \langle \mathcal{A}(\tilde{\f q})  , \begin{pmatrix}
\tu - \f u \\\tet-\theta
 \\ \tilde{\mu}-\mu 
 \end{pmatrix} \right \rangle  +  M \| \varphi - \tp \|_{( W^{1,\infty}(\Omega))^*} \| \mathcal{A}_3(\tilde{\f q})    \|_{( W^{1,\infty}(\Omega))^*}  \right )  e^{\int_s^t  \mathcal{K} (\tilde{\f q}) \de \tau } \de s \,,\label{relenenergy}
\end{multline}
where $\mathcal{A}_3$ denotes  the third line of~\eqref{operatorA}, \textit{i.e.}, the left-hand side of equation~\eqref{eq2}$_1$ for the test function $\tilde{\f q}\in \mathbb{Y}$.   
The idea for the definition of  dissipative solutions is to replace the weak formulation of the momentum balance~\eqref{weakmomentum}, the energy inequality~\eqref{energyin}, and entropy inequality~\eqref{entropy} by the above relative energy balance for all reasonable test functions. 
 
\begin{definition}[dissipative solution]\label{def:diss} Let $\delta =0$. 
A triple $ (\f u, \theta, \varphi
 ) \in \mathbb{X}$ and a chemical potential $\mu  \in L^2(0,T;W^{1,1}(\Omega))
  $ is called a dissipative solution to~\eqref{eq}, if 
\begin{align*}
\frac{|(\nabla \f u)_{\sym}| }{\sqrt{\theta}}\,,  &  \frac{|\nabla \mu |}{\sqrt{\theta}} \in L^2(0,T;L^2(\Omega)) \,,
\end{align*}
   the incompressibility condition $ \di \f u =0$ is fulfilled a.e.~in $\Omega \times (0,T)$, as well as  
   the relations~\eqref{phaseeq},~\eqref{mueq}, and~\eqref{weakentro}. 
   Additionally, the relative energy inequality~\eqref{relenenergy} is fulfilled for a.e.~$t\in (0,T)$ and for all
  $(\tu,\tet,\tp
  )\in \mathbb{Y}$,
where the chemical potential $\tilde{\mu}$ is given by $ \tilde{\mu} =-\varepsilon \Delta \tp +\frac{1}{\varepsilon} F'(\tp) - \tet $ such that $ \f n \cdot \nabla \tilde{\mu} = 0$ on $\Omega \times (0,T)$. 
\end{definition} 
%
%

 \begin{theorem}[Existence of  dissipative solutions]\label{thm:diss}
 Let $ \Omega \subset \R^d$ be sufficiency smooth%
, $\delta =0$,  and Hypothesis~\ref{hypo} be fulfilled. To every $\theta_0 \in L^1(\Omega) $ with $ \ln  \theta_0  \in L^1(\Omega)$ and $ \theta _0 (\f x ) >0 $ a.e.~in $\Omega$ as well as $ \varphi_0 \in H^1(\Omega)$ with $F(\varphi_0)\in L^1(\Omega)$ and $\f u_0\in \f H$, there exists at least one dissipative solution according to Definition~\ref{def:diss}.

Especially, every weak solution according to Definition~\ref{def:weak} is a dissipative solution according to~\ref{def:diss} by the natural identification (without the defect measure $m$).

 \end{theorem}
 
\begin{remark}[Dissipative solutions and regularity]\label{rem:disreg}
We want to argue that dissipative solutions are a reasonable solution concept. First, they comply with the minimal assumptions on a reasonable solution concept due to Lions~\cite[Sec.~4.4]{lionsfluid},\textit{ e.g.},
these solutions exists (as the previous theorem asserts) and they fulfill the so-called \textit{weak-strong uniqueness} criterion. This means that, in case that there exists a weak solution $\tilde{\f q}$ fulfilling the additional regularity properties, \textit{i.e.}, $\tilde{\f q}\in \mathcal{Y}$, then every dissipative solution emanating from the same initial datum coincides with this solution. 
This property follows directly from the relative energy inequality~\eqref{relenenergy}. Indeed, inserting the weak solution $\tilde{\f q}
 =( \tu,\tet,\tp)$ into~\eqref{relenenergy}, the right-hand side vanishes, if the initial values match. This implies that also the left-hand side has to be zero, \textit{i.e.}, every dissipative solution coincides with $\tilde{\f q}$. 

On the other hand, it also holds that if there exists a regular dissipative solution, then this solution is a regular weak solution, \textit{i.e.}, a strong solution. 
Indeed, assume that the dissipative solution $\bar{\f q}$ is regular, \textit{i.e.}, $\bar{\f q} = ( \bar{\f u}, \bar{\theta}, \bar{\varphi}) \in \mathcal{Y}$, then also $ \tilde{\f q} = \bar{\f q} + \alpha \f r \in \mathcal{Y}$ for $\f r= (\f r_{\f u}, \f r _{\theta}, \f r_{\varphi})^T\in \C^\infty_c(\ov\Omega \times [0,T] ; \R^5)$ with $\alpha > 0 $ sufficiently small. 
First, we observe that due to the additional regularity, the relations~\eqref{phaseeq} and~\eqref{mueq} hold a.e.~pointwise. This implies that the third component of~$\mathcal{A}(\bar{\f q})$ vanishes. 
Inserting $ \tilde{\f q} = \bar{\f q} + \alpha \f r $  into~\eqref{relenenergy} for the dissipative solution $\f q = \bar{\f q}$ and dividing by $\alpha$, we end up with
\begin{align*}
o(\alpha) \leq \int_0^t    \left \langle \mathcal{A}(\bar{\f q})  , \begin{pmatrix}
\f r 
 \end{pmatrix} \right \rangle e^{\int_s^t  \mathcal{K} (\tu,  \tet,\tp) \de \tau+ o(\alpha) } \de s  + o(\alpha) \,,
\end{align*}
where $o (\alpha ) \ra 0 $ for $\alpha \ra 0$,
since the only linear term in $\alpha$ occurs in the last term on the right-hand side of~\eqref{relenenergy} and all other appearing terms are super-linear in $\alpha$. 
Passing to the limit $\alpha\ra 0$ implies that 
\begin{align*}
0 \leq  \int_0^t \left \langle \mathcal{A}(\bar{\f q}) , 
\f r\right \rangle \de s ={}& \int_0^t \left (  \t\bar{\f u}+(\bar{\f u}\cdot\nabla)\bar{\f u} - \di ( \nu(\bar{\theta})\nabla \bar{\f u} ) + \varepsilon \di ( \nabla \bar{\varphi}\otimes \nabla \bar{\varphi}) , \f r _{\f u} \right ) \de s \\ &
+ \int_0^t 
\left ( (\t  +( \bar{\f u} \cdot\nabla )) (- \ln  \bar{\theta} - \bar{\varphi} ) +  \di ( \kappa (\bar{\theta})\nabla \ln \bar{\theta}) 
, \f r_{\theta}\right  ) \de s \\ &
+ \int_0^t 
\left (  \kappa(\bar{\theta}) | \nabla \ln \bar{\theta}|^2 + \nu(\bar{\theta}) \frac{| (\nabla \bar{\f u} )_{\sym}|^2 }{ \bar{\theta}}+ \frac{ | \nabla \bar{\mu}|^2 }{\bar{\theta}}
, \f r_{\theta}\right  ) \de s 
 \end{align*}
the above inequality is in fact an equality  (since $\f r$ was arbitrary) and hence, $\bar{\f q}$  fulfills a standard weak formulation.  To find the above equality, we inserted the definition of $\mathcal{A}$ and used, that the last entry vanishes. Note that $\bar\mu$ is defined in the usual way, \textit{i.e.}, $\bar \mu = - \varepsilon \Delta \bar\varphi + F'( \bar \varphi ) - \bar \theta $ with $\f n \cdot \nabla \bar\mu = 0$ on $\partial\Omega \times (0,T)$. 
\end{remark}
\begin{remark}[formulation of dissipative solutions]\label{rem:form}
During the proof of the previous remark, it became obvious that the property of \textit{regularity implies uniqueness} as well as \textit{weak-strong uniqueness}  holds without the relations~\eqref{entropy} and~\eqref{energyin}.
That is, why we do not incorporate them into the Definition~\ref{def:diss}. This would be possible and they are also weakly-sequential stable with respect to the underlying topology.
Both formulations,~\eqref{entropy} and~\eqref{energyin}, are excluded from Definition~\ref{def:diss}, since they do not seem to contribute much information. 
The energy inequality~\eqref{energyin} follows from~\eqref{relenenergy} by choosing~$(\tu, \tet, \tp)= (0,0,0)$. 
The entropy inequality~\eqref{entropy} for constant test functions follows by choosing~$(\tu, \tet, \tp)= (0,\bar\theta,0)$ in ~\eqref{relenenergy}  and formally passing to the limit~$\bar\theta \ra \infty$. 
The weak-strong uniqueness result and the regularity implies uniqueness result from Remark~\ref{rem:disreg} even holds without the relation~\eqref{weakentro}, but without this relation, we lack any control on the time derivative of the temperature or rather the entropy. This would deprive us from the possibility to establish strong convergence of approximate temperatures, which in turn 
 would lead to Young measure-valued temperatures in~\eqref{relenenergy}. Therefore, we kept~\eqref{weakentro} in Definition~\ref{def:diss}.
In order to establish strong convergence of the temperature. We remark that it would be enough, to include some estimate of the time derivative into the formulation, like
\begin{align*}
\| \t \ln \theta \| _{\mathcal{M}([0,T]; (W^{1,p}(\Omega ))^*)}  \leq c \,
\end{align*}
in order to establish strong convergence of the temperatures and such a formulation does not rely on any measure-valued relations. 
\end{remark}

An underlying idea of dissipative solution is that no equation is fulfilled anymore. This may seems odd, that a solution to a partial differential equation is given as an inequality. 
%
%
 But firstly, inequalities serve as a reasonable solution concept in the context of Gradient flows~\cite{mielke}, in the form of De Giorgis upper dissipation distance. In this regard, dissipative solutions may also be interpreted as a generalization of the concept of minimizing movements~\cite{giorgi} applicable to more general  non-gradient systems (\textit{e.g.}, GENERICs~\cite{generic}). 

Secondly, away from a certain regularity regime, the equations may not be considered as a  good model. There has been extensive work on the non-uniqueness of weak solutions using convex-integration techniques (see~\cite{isett} or~\cite{buckmaster}). Additionally, the equations are derived from energetic principles often under the assumption of certain smoothness. 
Then the question arises, why should a generalized solution concept even fulfill the equations in a distributional sense, even though they may not lead to  a solution complying to the overall energetic principles. 
%
%

The concept of dissipative solutions follows another approach and compares the dissipative solution to smooth solutions, which fulfill the equations only approximately, but inherit enough regularity to deduce uniqueness, \textit{i.e.}, are elements in a regularity class for which the equations make sense. 
Even though dissipative solution comply with the underlying energy dissipation relations, they are still far from being unique. Therefore additional selection criteria are needed in order to choose a good solution within this class of dissipative solutions. The concept of maximal dissipation, \textit{i.e.}, selecting the solution dissipating the most energy
was proposed (see~\cite{dafermos} or~\cite{breit}) as a selection criterion to identify the physically relevant solution.
This implies that the dissipative solution with minimal energy should be selected. 
Therefore, one may considers the minimization problem (compare to~\cite{maxdiss})
\begin{align*}
\min _{\{\f q \in \mathcal{X}  \}}  \int_0^T \mathcal{E}(\f q(t)) \de t \quad \text{s.t.}\quad \f q \text{ is a dissipative solution according to Definition~\ref{def:diss}}\,.
\end{align*}
Including the inequality conditions via a Lagrangian multiplier and taking the supremum over all side conditions, we end up with the definition of maximal dissipative solutions  (compare to~\cite{maxdiss}). 
Due to the fact that the energy $\mathcal{E}$ is nonvonvex, this selection criterion still grants no unique solution. This stems from the fact that the cost functional $\mathcal{E}$ is not convex, but also the set of dissipative solutions is not convex.

In contrast, it is known that the associated energy in the limit of $\varepsilon \ra 0$ is convex. This energy is the Hausdorff measure of the resulting sharp interface $\Gamma $, $\mathcal{H}^{d-1} ( \Gamma) = \| \nabla \chi \|_{\mathcal{M}(\Omega)}$, where $\chi $denote the indicator function associated to one fluid species~\cite{abels2}. Since this is a convex energy, there may be some hope that in the limit $\varepsilon\ra 0$, maximal dissipation even provides a unique solution. Such a cost functional will favor shorter (and therewith smoother) interfaces, which is desirable in applications~\cite{checkerboard}.

 \subsection{Preliminaries}
We collect different lemmata that are helpful in the remainder of the article.
The following result was already used in~\cite{larosc}. 
\begin{lemma}\label{lem:lnest}
Let $\{a_n\}$, $\{ b_n \} \subset L^1(\Omega\times (0,T) )$ and $b_n > 0 $ a.e.~in $\Omega\times (0,T)$ for all $n \in \N$. Assume that there exists a constant $c$ such that
\begin{align}
\int_0^T\int_\Omega \frac{a_n^2}{b_n} \de \f x \de t  \leq c \quad\text{and}\quad \int_0^T \int_\Omega b _n\ln b_n \de \f x \de t  \leq c \,.\label{orliszest}
\end{align}
Then there exists a constant such that 
\begin{align*}
\int_0^T\int_\Omega  | a_n|  \ln ^{1/2} (1+|a_n|) \de \f x \de t \leq c \,. 
\end{align*}

\end{lemma}
\begin{proof}

In order to infer a bound for $\{a _n\}$, we 
consider a new convex function $\psi:\R_+\to[1,+\infty]$ defined as 
$\psi(r)=(1/4) (r^2(2\log r -1)+1)$, where it is intended that $\psi(1)=0$ and $\psi(r)\equiv +\infty$ as $r<1$. 
Determining the precise expression of the conjugate function $\psi^*(s)$ is difficult,
but we can at least estimate it appropriately. We recall
that 
\begin{equation*}
  \psi^*(s) = \max_{r\in \R_+} \big( sr - \psi (r) \big)
\end{equation*}
and a simple computation shows that the maximum is attained at $s =r\log r$.
Hence, if $r$ is the maximizer, using first that $s =  r \log r$ 
and then that $s+1\le r^2$ (which holds as $r\geq 1$),
we have
\begin{align}\nonumber
  \psi^*(s) & =  r^2 \log r - \psi(r) = \frac{1}{2} r^2 \log r + \frac{1}{4} r^2 -\frac{1}{4} 
  = \frac{s^2}{\log r^2} + \frac{s^2}{\log^2 r^2} - \frac{1}{4}\nonumber  \\
  &\leq \frac{s^2 }{\log (s+1)} + \frac{s^2}{\log^2 (s+1) } - \frac{1}{4}\nonumber\,.
\end{align}
Additionally, we observe for $\psi^*$ that for any $y\in [2,\infty)$ it holds
\begin{align}
 \psi^*( y \log ^{1/2} y ) \leq \frac{y^2 \log  y}{\log\left  (1+y \log ^{1/2} y\right )}
 + \frac{y^2 \log  y}{\log^2\left  (1+y \log ^{1/2} y\right )} - \frac{1}{4} \leq c (y^2 + 1)\,,\label{squarebound}
\end{align}
since the function
\begin{align*}
  y \mapsto \frac{ \log  y}{\log\left  (1+y \log ^{1/2} y\right )}+ \frac{ \log  y}{\log^2\left  (1+y \log ^{1/2} y\right )} 
\end{align*}
is bounded for $y \in [2 ,\infty)$, which is obvious for any compact subset in $[2,\infty)$ 
and also holds as $y\nearrow \infty$ as an easy check shows. 

Now, setting for simplicity $u :=  \sqrt{b}+1$ and  $v:=a+2 u $,
we have 
\begin{align}\nonumber
  v \log^{1/2} v 
   & = \frac{v}{u}
     \bigg[ \log \Big( \frac{v}{u} \Big) + \frac12 \log(u^2 ) \bigg]^{1/2}u\\
 \nonumber
   & \le \frac{v}{u}
     \bigg[  \log^{1/2} \Big( \frac{v}{u} \Big) + \frac{1}{\sqrt 2}\log^{1/2}( u^2  ) \bigg] u  \\
 \nonumber
  & \le \frac{v}{u}
       \log^{1/2} \Big( \frac{v}{u} \Big)  u +  \frac{1}{\sqrt 2}\frac{v}{u} u \log^{1/2}( u^2  ) \\
\nonumber
  & \le \psi^*\left (        \frac{v}{u}
       \log^{1/2} \Big( \frac{v}{u} \Big)\right ) + \psi(u) + \frac{1}{\sqrt 8}\left ( \frac{v^2}{u^2} +  u^2 \log (u^2)   \right )    \\
  \nonumber
    &  \leq c \left (  \frac{v^2}{u^2}  +1\right ) + \frac{1}{4} u^2 (2 \log (u^2) -1 ) + \frac14 + \frac{1}{\sqrt 8}\left ( \frac{v^2}{u^2} +  u^2 \log (u^2) \right )\\
   &\leq c  \left (      \frac{v^2}{u^2}+u^2 \log (u^2)     +1 \right )
    \,,
 \label{conj14-2}
\end{align}
where we used calculation rules for the logarithm, properties of the square root under the additional observation that $\log (v/u)\geq 0$,
the {Legendre--Fenchel--Young} inequality as well as the standard {Young}'s inequality, and~\eqref{squarebound} as well as the definition of $\psi$. 
Finally, integrating \eqref{conj14-2} over $\Omega\times (0,T)$, 
we observe that the right-hand side is bounded due to \eqref{orliszest}.
From the left-hand side, we deduce with the bound~\eqref{orliszest} that
\begin{align}
  \int_0^T \int_\Omega |  a_n | \log^{1/2} \big( 1 + |a_n | \big) \de \f x\de t 
   \le c.
   \label{boundphit}
\end{align}
\end{proof}

\begin{lemma}\label{lem:relpos}
Let  $ \delta \in [0,1)$ and $ \beta \in [2\delta ,2-2\delta]$. For $\tet>0$, it holds
\begin{align*}
\theta - \tet - \left ( f''_\delta(\tet)\right )^{-1} \left ( f'_\delta(\theta) - f'_\delta(\tet)\right ) \leq  \Lambda_\delta ( \theta | \tet) \,
\intertext{as well as}
 \tet^{1-\beta/2} \left ( \theta ^{\beta/2} - \tet ^ { \beta/2} - \frac{\beta}{2} \tet^{\beta/2-1}
  \left ( f''_\delta(\tet)\right )^{-1} \left ( f'_\delta(\theta) - f'_\delta(\tet)\right )
 \right) 
 \leq   \Lambda_\delta  ( \theta | \tet)
 \intertext{and if additionally $\beta \in [4\delta , 1 -\delta]$, it holds}
 ( \theta^{\beta/2} - \tet ^{\beta/2}) ^2 \leq c \Lambda _\delta ( \theta | \tet) \,,
\end{align*}
where the constant depend on $\tet$, \textit{i.e.}, its lower bound. 
\end{lemma}
\begin{proof}
First, we observe for $\delta = 0$ that
\begin{align*}
\Lambda_0(\theta | \tet) = - \theta ( \log \theta - 1) + \tet ( \log \tet - 1) + \log \theta ( \theta - \tet ) = \theta - \tet - \tet ( \log \theta - \log \tet ) \,.
\end{align*}
and similar for $\delta >0$ that
\begin{align*}
 \Lambda_\delta ( \theta | \tet) ={}& - \frac{1}{\delta(\delta+1)} \left ( \theta ^{\delta+1} - \tet^{\delta +1} - (\delta +1)\theta^\delta ( \theta - \tet ) \right ) = \frac{1}{\delta+1} \left ( \theta^{\delta+1} - \tet \theta^\delta - \frac{1}{\delta} \tet ( \theta ^ \delta - \tet^\delta ) \right ) \\={}& \frac{1}{\delta+1 } \left ( ( \theta^\delta - \tet^\delta )( \theta - \tet) + \tet^\delta \left ( \theta - \tet - \frac{1}{\delta} \tet^{1-\delta } ( \theta^\delta - \tet^\delta ) \right )\right )\,.
\end{align*}
Note that both terms on the right-hand side are positive, the first one since $\delta>0$ and the second one since the function $s \mapsto s ^{1/\delta}$ is convex. 
Thus, we proved the first assertion.

For the second one, we first find for $\delta=0$ that
\begin{align*}
\tet^{1-\beta/2}& \left ( \theta ^{\beta/2} - \tet ^ { \beta/2} - \frac{\beta}{2} \tet^{\beta/2}(\log \theta - \log \tet)\right) \\
\leq {}&
\tet^{1-\beta/2} \left ( \theta ^{\beta/2} - \tet ^ { \beta/2} - \frac{\beta}{2} \tet^{\beta/2}(\log \theta - \log \tet)\right) 
\\
&+ \tet^{\beta/2} \left ( \theta ^{1-\beta/2} - \tet ^ { 1-\beta/2} - \left (1-\frac{\beta}{2}\right ) \tet^{1-\beta/2}(\log \theta - \log \tet)\right) 
\\
&+ (  \theta ^{\beta/2} - \tet ^ { \beta/2}) ( \theta ^{1-\beta/2} - \tet ^ { 1-\beta/2}) 
\\
={}&  \theta - \tet -\tet(\log \theta - \log \tet) \,.
\end{align*}
The inequality holds, since both added terms are non-negative  for $\beta\in(0,2)$ and the equality follows from calculating the terms explicitly. For $\beta =2$, there is nothing to show. 
Similarly, we observe for $\delta>0$ that
\begin{align*}
\tet^{1-\beta/2}& \left ( \theta ^{\beta/2} - \tet ^ { \beta/2} - \frac{\beta}{2\delta } \tet^{\beta/2-\delta}( \theta^ \delta - \tet^ \delta )\right) \\
\leq {}&
\tet^{1-\beta/2} \left ( \theta ^{\beta/2} - \tet ^ { \beta/2}- \frac{\beta}{2\delta } \tet^{\beta/2-\delta}( \theta^ \delta - \tet^ \delta )\right) 
\\
&+ \tet^{\beta/2} \left ( \theta ^{1-\beta/2} - \tet ^ { 1-\beta/2} - \frac{1}{\delta } \left (1-\frac{\beta}{2}\right ) \tet^{1-\beta/2- \delta }(\theta^ \delta - \tet^ \delta)\right) 
\\
&+ (  \theta ^{\beta/2} - \tet ^ { \beta/2}) ( \theta ^{1-\beta/2} - \tet ^ { 1-\beta/2}) 
\\
={}&  \theta - \tet - \frac{1}{\delta} \tet^{1-\delta } ( \theta^\delta - \tet^\delta ) \,.
\end{align*}
For the last assertion, we observe in the case $\beta \in  [2\delta , 1 -\delta ]$  that
\begin{align*}
 ( \theta^{\beta/2} - \tet ^{\beta/2}) ^2 ={}& \theta^\beta - \tet ^\beta - \beta \tet^{\beta-1}  \left ( f''_\delta(\tet)\right )^{-1} \left ( f'_\delta(\theta) - f'_\delta(\tet)\right )
\\& -2 \tet^{\beta/2}\left ( \theta ^{\beta/2} - \tet ^ { \beta/2} - \frac{\beta}{2} \tet^{\beta/2-1}
  \left ( f''_\delta(\tet)\right )^{-1} \left ( f'_\delta(\theta) - f'_\delta(\tet)\right) \right )\,. 
\end{align*}
Since both terms on the right-hand side may be estimated according to assertion two, assertion three follows. 
\end{proof}

\begin{lemma}\label{lem:log}
Assume that $\theta \in L^\infty(0,T;L^\delta (\Omega))$ such that $ \| \theta \|_{L^\infty( L^\delta (\Omega))} +\| f'_\delta (  \theta)  \|_{L^\infty( L^1(\Omega))} \leq c$ with $\theta(\f x) > 0$ for a.e.~$\f x\in\Omega$ and $ \tet \in L^\infty(\Omega\times (0,T))$ with $\tet (\f x) \geq \underline{\tet} >0$  for a.e.~$\f x\in\Omega$. 
Then there exists a constant $C>0$ such that 
\begin{align*}
 \| f'_\delta (\theta)  -  f'_\delta (\tet) \|_{L^1(\Omega)} ^2 \leq C \int_\Omega \Lambda _\delta  ( \theta | \tet)  \de \f x \,,
\end{align*}
where $C$ only depends on $c$, $ \| \tet \|_{L^\infty(\Omega \times (0,T))}$, and $\underline{\tet}$. In the case of $\delta >0 $ the constant also depends on $\underline{\theta}$. 
\end{lemma}
\begin{proof}
For this proof, we have to distinguish the two possibilities $ \theta \geq \tet $ and $\theta< \tet $. 
In the first case, we consider the function $ h_\delta  : [0,\infty)\ra [0,\infty)$ given by $ h_0(x) = \exp(x )-1-x$ for $\delta = 0$ and $ h_\delta (x) = (x+1)^{\frac{1}{\delta}} -1 - \frac{1}{\delta }x   $. This functions $h_\delta$ are strictly monotone increasing, bijective, and convex. Hence, their inverse functions $h_\delta^{-1}  : [0,\infty)\ra [0,\infty)$ exists and is strictly monotone increasing, bijective, but in contrast concave. This facts can be observed by $ 1 = \partial_x \left ( h_\delta^{-1}(h_\delta (x)) \right ) =( h_\delta^{-1})'(h_\delta (x))h_\delta '(x) $ and $0=  \partial_{xx}^2 \left ( h_\delta ^{-1}(h_\delta (x)) \right ) = ( h_\delta ^{-1})''(h_\delta(x))(h_\delta '(x))^2+ ( h_\delta ^{-1})'(h_\delta (x))h_\delta ''(x)$.
In the second case, we consider 
$ \tilde h_\delta  : [0,a_\delta)\ra [0,\infty)$, where $a_0 = \infty$ and $a_\delta = 1$ for $\delta >0$,  given by  $ \tilde h_\delta (x) = h_\delta (-x) $ such that $ \tilde h_0 (x) = \exp( - x) - 1 + x $ and $\tilde h_\delta (x) = ( 1- x)^{1/\delta} - 1 +\frac{1}{\delta } x $ for $\delta >0$. 
As beforehand, it can be seen that $ \tilde{ h}_\delta $ is strictly monotone increasing, bijective, and convex.

Additionally, we need some lower bound for $h^{-1}_\delta$ and $\tilde{h}_\delta^{-1} $  in the sense of $ h^{-1}_\delta ( y) \geq  c \sqrt y$ and vice verse for $\tilde h_\delta ^{-1}$. 
Therefore, we consider the Taylor expansion of the exponential function
\begin{align*}
e^x = 1 + x +\frac{1}{2} \int_0^1 e^{sx} \de s  x^2 \quad\Rightarrow \quad  e^x - 1 - x \geq \frac{1}{2} x^2 \quad\Rightarrow \quad  x \geq h_0^{-1} \left ( \frac{1}{2} x^2\right  ) \,.
\end{align*}
Similarly, we find for  $ \tilde h_\delta ^{-1} $ 
\begin{align*}
e^{-x }= 1 - x +\frac{1}{2} \int_0^1 e^{-sx} \de s  x^2 \quad\Rightarrow \quad  e^{-x }- 1 + x \leq  \frac{1}{2} x^2 \quad\Rightarrow \quad  x \leq \tilde h_0^{-1} \left ( \frac{1}{2} x^2\right  ) \,.
\end{align*}
In the case $\delta =0 $, we may deduce, $h^{-1}_\delta ( y) \leq   c \sqrt y$ as well as  $\tilde h_\delta ^{-1}(y)\geq  c \sqrt y $ with $$ \lim_{y \ra 0 } (\tilde h_\delta ^{-1}(y))/ \sqrt {y} = c >0\,.$$
Similar assertions hold in the case $\delta >0$.
%
%
%
%
We use Jensen's inequality for concave functions: For $ g : [0,\infty)\ra [0,\infty)$ concave it holds
\begin{align*}
\int_\omega g ( f (\f x) ) \de \f x \leq g \left ( \int_\Omega f (\f x) \de \f x \right ) \text{ for all measurable } \omega \subset \Omega\,
\end{align*}
to connect the results, we define 
\begin{align*}
\omega := \{ \f x \in \Omega | \theta(\f x) \geq \tet(\f x)\} \quad \text{and} \quad \tilde{\omega}=\Omega/ \omega 
\end{align*}
and observe in the case $ \delta =0$ 
\begin{align*}
\int_\Omega \left | \ln \theta - \ln \tet  \right | \de \f x ={}& \int_\omega \log \frac{\theta}{\tet} \de \f x + \int_{\tilde{\omega}}\log \frac{\tet }{\theta}\de \f x 
\\
={}&\int_\omega h_\delta^{-1} \left ( h_\delta \left (  \log \frac{\theta}{\tet}\right )\right ) \de \f x + \int_{\tilde{\omega}}\tilde{h}_\delta ^{-1} \left ( \tilde{h}_\delta \left (\log \frac{\tet }{\theta}\right )\right )\de \f x 
\\ 
\leq {}&h_\delta^{-1} \left ( \int_\omega  h_\delta\left (  \log \frac{\theta}{\tet}\right ) \de \f x\right ) +\tilde{h}_\delta^{-1} \left ( \int_{\tilde{\omega}} \tilde{h}_\delta\left (\log \frac{\tet }{\theta}\right )\de \f x \right )
\,
\end{align*}
and in the case $\delta>0$ 
\begin{align*}
\int_\Omega \left | f_\delta ' ( \theta ) - f'_\delta (\tet) \right | \de \f x ={}& f'_\delta (\tet)   \int_\omega \frac{f_\delta ' ( \theta ) }{ f'_\delta (\tet) }-1  \de \f x + f'_\delta (\tet)  \int_{\tilde{\omega}}1- \frac{f_\delta ' ( \theta ) }{ f'_\delta (\tet) }\de \f x 
\\
={}&\int_\omega h_\delta^{-1} \left ( h_\delta \left (  \frac{f_\delta ' ( \theta ) }{ f'_\delta (\tet) }-1 \right )\right ) \de \f x + \int_{\tilde{\omega}}\tilde{h}_\delta ^{-1} \left ( \tilde{h}_\delta \left (1- \frac{f_\delta ' ( \theta ) }{ f'_\delta (\tet) }\right )\right )\de \f x 
\\ 
\leq {}&h_\delta^{-1} \left ( \int_\omega  h_\delta\left (  \frac{f_\delta ' ( \theta ) }{ f'_\delta (\tet) }-1  \right ) \de \f x\right ) +\tilde{h}_\delta^{-1} \left ( \int_{\tilde{\omega}} \tilde{h}_\delta\left (1- \frac{f_\delta ' ( \theta ) }{ f'_\delta (\tet) }\right )\de \f x \right )
\,.
\end{align*}
The functions $h_\delta$ and $\tilde{h}_\delta$ are  chosen in a way that that $$ \tet  h_0 ( \log ( \theta/\tet) ) = \Lambda_0  ( \theta | \tet) = \tet \tilde{h}_0 ( \log ( \tet /\theta))\, $$
and 
$$ \tet  h_\delta \left  ( \frac{f_\delta ' ( \theta ) }{ f'_\delta (\tet) }-1\right ) =  \tet \tilde{h}_\delta  \left ( 1- \frac{f_\delta ' ( \theta ) }{ f'_\delta (\tet) }\right )
\leq \Lambda_\delta   ( \theta | \tet) \,, $$ respectively.
In the case of $h_\delta$, we immediately observe
\begin{align*}
h_\delta^{-1} \left ( \int_\omega \frac{1}{\tet}\Lambda_\delta ( \theta | \tet)  \de \f x\right ) \leq c \sqrt{\int_\omega  \frac{1}{\tet}\Lambda_\delta ( \theta | \tet)  \de \f x}\,.
\end{align*}


Due to the assumptions,  the relative energy is bounded, \textit{i.e.}, 
\begin{align*}
\int_\Omega \Lambda_\delta  ( \theta | \tet)  \de \f x \, \in \, L^\infty
(0,T)\,.
\end{align*}
Concerning the function   $ \tilde h_\delta^{-1} $, we found that it is bounded on bounded sets and that $ \tilde h_\delta^{-1} ( y) / \sqrt y \geq c $.
Combining this, we find 
\begin{align*}
\tilde{h}_\delta^{-1} \left ( \int_{\tilde{\omega}} \frac{1}{\tet }\Lambda_\delta (\theta | \tet) \de \f x \right )  = \frac{\tilde{h}_\delta^{-1} \left ( \int_{\tilde{\omega}} \Lambda_\delta (\theta | \tet) \de \f x \right )}{\sqrt{ \int_{\tilde{\omega}} \Lambda_\delta (\theta | \tet) \de \f x }}\sqrt{ \int_{\tilde{\omega}} \Lambda_\delta (\theta | \tet) \de \f x }\,,
\end{align*}
where the first factor on the right-hand side is just a constant. 
The fact that $\Lambda_\delta ( \theta | \tet) \geq 0$  implies $ \int_\omega \Lambda_\delta ( \theta | \tet) \de \f x \leq \int_\Omega \Lambda_\delta ( \theta | \tet) \de \f x $ for $ \omega \subset \Omega$. 
We find that 
\begin{align*}
\left\| f_\delta ' ( \theta ) - f'_\delta (\tet) \right\|_{L^1(\Omega)} \leq c \sqrt{\left \| \frac{1}{\tet}\Lambda_\delta  ( \theta | \tet) \right \|_{L^1(\Omega)}}\,.
\end{align*}

This implies the assertion.  Note that $ \tet$ is essentially bounded by positive constants from below and above.
\end{proof}

\begin{lemma}\label{lem:diff}
Let $\beta \in [ 4 \delta , 2-2\delta]$. Assume that $\theta \in  L^\infty(0,T;L^1(\Omega))$ such that $ \| Q_\delta(\theta) \|_{L^\infty( L^1(\Omega))} +\|f'_\delta(  \theta) \|_{L^\infty( L^1(\Omega))} \leq c$ with $\theta(\f x) > 0$ for a.e.~$\f x\in\Omega$ and $ \tet \in L^\infty(\Omega)$ with $\tet (\f x) \geq \underline{\tet} >0$  for a.e.~$\f x\in\Omega$. Then there exists a constant $c>0$ such that
\begin{align*}
\| \theta^{\beta/2} - \tet^{\beta/2} \|_{L^1(\Omega)} ^2 \leq c \int_\Omega \Lambda_\delta ( \theta | \tet) \de \f x  \,. 
\end{align*}
where $C$ only depends on $c$, $ \| \tet \|_{L^\infty(\Omega \times (0,T))} $, and $\underline{\tet}$.

\end{lemma} 
\begin{proof}
First, we observe that
\begin{align*}
( \theta^{\beta/2} - \tet^{\beta/2} ) = ( \theta ^{\beta/4} - \tet ^{\beta/4})^2 - 2 \tet^{\beta/4} ( \theta ^{\beta/4}- \tet ^{\beta/4})\,,
\end{align*}
such that we find by Young's inequality and Jensen's inequality
\begin{align*}
\| \theta^{\beta/2} - \tet^{\beta/2} \|_{L^1(\Omega)} ^2 \leq{}& \left ( \|  \theta ^{\beta/4} - \tet ^{\beta/4} \| _{L^2(\Omega)}^2 + 2 \|  \tet^{\beta/4}  \|_{L^\infty(\Omega)} \| \theta ^{\beta/4}- \tet ^{\beta/4} \| _{L^1(\Omega)} \right )^2 
\\
\leq {}& 2 \left (  \|  \theta ^{\beta/4} - \tet ^{\beta/4} \| _{L^2(\Omega)}^4 + 4 \|  \tet^{\beta/4}  \|_{L^\infty(\Omega)}^2 \| \theta ^{\beta/4}- \tet ^{\beta/4} \| _{L^1(\Omega)}^2  \right )
\\
\leq {}& 2 \left (  \|  \theta ^{\beta/4} - \tet ^{\beta/4} \| _{L^2(\Omega)}^4 + 4c  \|  \tet^{\beta/4}  \|_{L^\infty(\Omega)}^2 \| \theta ^{\beta/4}- \tet ^{\beta/4} \| _{L^2(\Omega)}^2  \right )\,.
\end{align*}
Lemma~\ref{lem:relpos} implies that
\begin{align*}
\| \theta^{\beta/4} - \tet^{\beta/4} \|_{L^1(\Omega)} ^2 \leq c  \left ( \left ( \int_\Omega \Lambda_\delta ( \theta | \tet) \de \f x \right )^2  + \|  \tet^{\beta/4}  \|_{L^\infty(\Omega)}^2  \int_\Omega \Lambda_\delta ( \theta | \tet) \de \f x \right )\,.
\end{align*}
Since the relative energy is bounded, \textit{i.e.}, 
\begin{align*}
\int_\Omega \Lambda_\delta ( \theta | \tet)  \de \f x \, \in \, L^\infty
(0,T)\,,
\end{align*}
we find the assertion.

\end{proof}
\begin{lemma}\label{lem:fenchel}
Let $ \theta, \tet ,g: \Omega \ra \R_+$ with $ g( \f x) \leq 1/2$ for a.e.~$\f x \in \Omega $. Then it holds that
\begin{align*}
\int_\Omega \frac{\theta}{\tet} g \de \f x \leq  \int_\Omega \frac{1}{\tet} \Lambda_\delta ( \theta| \tet) + 2^{\frac{1}{1-\delta}} g \de \f x \,.
\end{align*}
\end{lemma}
\begin{proof}
We define the convex function $\psi_\delta : \R_+ \ra \R_+ $ via $x \mapsto x-1 - \ln x $ in the of $\delta =0$ and $x \mapsto x-1-\frac{1}{\delta }( x^\delta  -1 ) $ in the case of $\delta \in (0,1)$.  
The convex conjugates are given by $\psi_0^* ( y) = - \ln (1-y) $ for $\delta = 0$  and $\psi_\delta^* ( y) = \frac{1-\delta}{\delta}\left ( ( 1- y)^{\delta/(\delta-1) }- 1 \right ) $. 
Computing the first derivatives of the conjugates, we observe 
$(\psi^*_\delta )'(y) = (1-y)^{-\frac{1}{1-\delta}}$. 
We may combine the results by
\begin{align*}
\frac{\theta}{\tet} g \leq \psi_{\delta}\left ( \frac{\theta}{\tet} \right ) + \psi_\delta^* ( g)  = \frac{1}{\tet} \left ( \theta - \tet  - ( f''_\delta ( \tet) )^{-1} ( f'_\delta (\theta)- f'_\delta (\tet) ) \right ) +  \psi_\delta^* ( 0)+ \int_0^g\left ( \psi_\delta^*\right )  '(s)\de s g \,.
\end{align*}
The assertion of the lemma follows by Lemma~\ref{lem:relpos}, $\psi^*_\delta (0)=0 $, inserting $(\psi^*_\delta )'$, and observing that $g( \f x) \leq 1/2$.

\end{proof}
 
 \section{Existence of measure-valued solutions\label{sec:ex}}
This section is devoted to the existence of measure-valued solutions.  
 \subsection{Approximate scheme}
 In this section, we present an approximate scheme. We only want to comment on 
the proof of existence of solutions to the  
  approximate scheme and do not prove it  in full detail, since this seems to be fairly standard.
  The scheme consists of a discretization and a regularization step. The Navier--Stokes-like equation and the phase-field equations are discretized by a Galerkin approach and the energy balance is regularized appropriately. 
  Let $V_n\subset \f V$ be a Galerkin space spanned by eigenfunctions of the Stokes-problem with homogeneous Dirichlet data and $W_n\subset H^1(\Omega)$ be a Galerkin space spanned by the eigenfunctions of the Laplace equation with homogeneous Neumann conditions (with $\int\varphi_n \de \f x =0$). Let $\gamma >0$, then we consider the approximate system
\begin{align}
\begin{split}
\left ( \t \f u_n + ( \f u_n \cdot \nabla )\f u_n , \f v \right ) + \left ( \nu (\theta_n ) \nabla \f u_n 
; \nabla \f v \right ) 
- \left ( \nabla \varphi_n \mu_n, \f v \right ) 
={}& 0 \text{ for all }\f v \in V_n \,,
\\
\left ( \t \varphi_n + ( \f u_n \cdot\nabla ) \varphi_n , \zeta\right ) + \left ( \nabla \mu_n , \nabla \zeta \right ) ={}&0 \text{ for all }\zeta \in W_n\,, \\
\varepsilon \left ( \nabla \varphi_n , \nabla \eta \right ) + \left (  \frac{1}{\varepsilon}F'(\varphi_n) -\theta _n- \mu_n , \eta \right ) ={}&0 \text{ for all } \eta \in W_n \,,\\
(\t+ (\f u_n \cdot \nabla))   Q(\theta_n ) + \theta _n  \Delta \mu_n - \di \left ( \kappa_\gamma ( \theta_n ) \nabla \theta_n \right ) +  \frac{\gamma}{\theta_n ^3 }= {}& \nu(\theta_n ) | \nabla \f u_n |^2 + | \nabla \mu_n |^2 \text{ in }\Omega \times (0,T) \,,
\\
\f n \cdot \kappa_\gamma (\theta_n) \nabla \theta_n = {}& 0 \text{ on } \partial \Omega \times (0,T) \,.
\end{split}\label{dis}
\end{align}  
  where $ \kappa_\gamma(r)= \kappa(r) + \gamma r^  p $ with $p\geq d^2$. It would be enough to assume that $ p> (d^2 - 4) /2d$, which is required to establish weak solutions (vanishing defect measure $m$, compare to Theorem~\ref{thm:existence},~\eqref{addreg}). The system is completed with appropriate initial values. 
  The existence to such an approximate system may be shown by Schauder's fixed point argument (see~\cite[Prop.~19]{indu} or~\cite[Sec.~3.4.3]{singular}), where it is essential to show the positivity of the temperature $\theta_n$ by a comparison principle  similar to~\cite[Lemma~17]{indu} or~\cite[Sec.~3.4.2]{singular}, where the regularization term, the last term on the left-hand side of the approximate energy balance, is essential to deduce non-negativity of the temperature on the approximate level. 
  It has to be taken into account that the elastic stress in the first equation of~\eqref{dis} was adapted such that the energy inequality holds on the discrete level. 
  First passing to the limit in the Galerkin discretization and afterwards in the regularization, we end up with the solution according to Definition~\ref{def:weak}. We refer to~\cite{indu} or~\cite[Sec.~3]{singular} for more details on such an approximation procedure. 
  
  The essential \textit{a priori} estimates and the weak-sequential stability to prove the existence result rigorously are given in the sequel of this section. 
 

 \subsection{\textit{A priori} estimates\label{apriori}}
\textbf{ Energy estimate.}  Formally, we deduce  by testing~\eqref{eq1} by $\f u$, equation~\eqref{eq3} by $1$, equation~\eqref{eq2}$_1$ by $\mu$ and equation~\eqref{eq2}$_2$ by $\t \varphi $,  adding all the resulting equations, and integrating over $\Omega \times (0,T)$ that
\begin{align*}
\int_\Omega \left ( \frac{1}{2 } | \f u |^2 + \frac{\varepsilon }{2}| \nabla \varphi |^2  +\frac{1}{\varepsilon} F(\varphi) + f_\delta(\theta )- \theta  f'_\delta (\theta)  \right ) \de \f x \leq c \quad\text{ for a.e. }t \in (0,T)\,.
\end{align*}
This implies due to the coercivity of $F$ (see Hypothesis~\ref{hypo}) the estimates 
\begin{align}
 \| \f u \|_{L^\infty(0,T;\f L^2(\Omega))} + \| \nabla \varphi \|_{L^\infty(0,T;\f L^2(\Omega))} + \esssup_{t\in(0,T)} \int_\Omega F(\varphi(t))\de \f x + \| Q_\delta (\theta)\|_{L^\infty(0,T;\f L^1(\Omega))}\leq C \,. \label{aprienergy}
\end{align}
The properties of $F$ let us conclude that additionally 
\begin{align*}
\| \varphi \|_{L^\infty(0,T;L^{p} (\Omega))}\leq c \quad \text{with } p <\infty \text{ for }d=2\text{ and } p =2d/(d-2)\text{ for }d\geq 3\,.
\end{align*}
\textbf{Entropy estimate.} Formally testing equation~\eqref{eq3} by $ 1/\theta$ leads to 
\begin{multline*}
\int_\Omega \left ( f'_\delta (\theta(t) ) - \varphi (t) \right ) \de \f x  +\int_0^t \int_\Omega \left ( \kappa(\theta ) | \nabla \ln \theta |^2 + \nu(\theta)  \frac{|(\nabla \f u)_{\sym}|^2 }{\theta }+ \frac{| \nabla \mu |^2 }{\theta }\right  )\de \f x \de s  =\\ \int_\Omega \left ( f'_\delta (\theta_0 ) - \varphi _0 \right ) \de \f x\,.
\end{multline*}
Note that the convection terms vanish due to the incompressibility of $\f u$.
For $\delta = 0$, we have to observe that $ \log \theta \leq \theta $ for $\theta \geq 1 $ in order to deduce again for all $\delta \in [0,1)$ that
\begin{align}
\| f'_\delta (\theta) \|_{L^\infty(0,T; L^1(\Omega))} + \| \nabla \hat{\kappa}(\theta) \|_{L^2(\Omega \times (0,T))} + \left \| \frac{\sqrt{\nu(\theta)}|\nabla \f u| }{\sqrt{\theta}} \right \|_{L^2(\Omega \times (0,T))} +  \left \| \frac{\nabla \mu  }{\sqrt{\theta}} \right \|_{L^2(\Omega \times (0,T))}  \leq c \,,\label{aprientro}
\end{align}
where Korn's inequality is used~\cite[Thm.~10.15]{singular}.
From Young's inequality, we observe for $\delta \leq \beta/2 $
\begin{align}
 \| \nabla f'_\delta (\theta)  \|_{L^2(\Omega \times (0,T))} \leq c  \| \nabla \hat{\kappa}(\theta) \|_{L^2(\Omega \times (0,T))} \label{nablafprime}
\leq c  \,.
 \end{align} 
\textbf{Additional estimates.}
Integrating~\eqref{eq2}$_1$ implies that the mass of $\varphi$ is conserved, \textit{i.e.}, $\int_\Omega \varphi (t) \de \f x = \int_\Omega \varphi(0)\de \f x $. 
Integrating~\eqref{eq2}$_2$, we find that $\int_\Omega \mu(t) \de \f x  = \int_\Omega (1/\varepsilon ) F'(\varphi(t)) - \theta(t) \de \f x \leq c $.
By Young's inequality, we may deduce that
\begin{align*}
\| \nabla \mu \|_{L^1(\Omega \times (0,T))} \leq \frac{1}{2} \int_\Omega \frac{|\nabla \mu|^2}{\theta} + \theta \de \f x \leq c \,,
\end{align*}
such that Poincar\'e's inequality implies 
\begin{align}
\| \mu\|_{L^{d/(d-1)}(\Omega \times (0,T))}\leq c \,.
\label{Nummer}
\end{align} 
\textbf{Estimate in the case $\delta =0$.}
Testing eqaution~\eqref{eq2}$_2$ by $\ln\theta$, we find
\begin{align}
\int_0^T\int_\Omega \theta \ln  \theta  \de \f x \de t \leq \| \nabla \varphi \|_{L^2(\Omega \times (0,T))} \| \nabla \ln \theta  \|_{L^2(\Omega \times (0,T))} + \int_0^T\int_\Omega  \frac 1\varepsilon F'(\varphi) \ln \theta  - \mu \ln \theta \de \f x \de t \,.\label{apri:theta}
\end{align}
Using the Legendre--Fenchel inequality for the convex conjugates $ y \mapsto e^y $ and $ x \mapsto x \ln x $, we estimate further 
\begin{multline*}
 \int_0^T \int_\Omega  F'(\varphi) \ln \theta  + \mu \ln \theta \de \f x \de t \\
 \leq c \left ( \| \theta \|_{L^1(\Omega\times (0,T))} + \int_0^T \int_\Omega \left ( \frac 1\varepsilon F'(\varphi) + \mu\right  ) \ln \left (\frac 1\varepsilon F'(\varphi) + \mu\right  ) \de \f x \de t\right )  \,,
\end{multline*}
where the right-hand side is bounded due to assumption~\eqref{growthF},~\eqref{aprienergy}, and~\eqref{Nummer}.
This implies a bound on $\theta \ln \theta $, because $\theta \ln \theta \geq - e^{-1}$ pointwise a.e.~in $\Omega \times (0,T)$. 
Applying Lemma~\ref{lem:lnest}, we may deduce the additional estimates
\begin{align}
\int_0^T\int_\Omega | \nabla \mu | \ln ^{1/2} (1+ | \nabla \mu| ) + \left |( \nabla \f u)_{\sym}\right  | \ln ^{1/2} \left (1+  \left |( \nabla \f u)_{\sym}\right  |\right ) \de \f x \de t \leq c \label{apri:diss}\,.
\end{align} 
 \textbf{Estimates for the time-derivative.}
 Comparison in equation~\eqref{eq1} implies by the above estimates that 
\begin{align*}
\int_0^T \| \t \f u \|_{ (\f W^{2,p}_{0,\sigma}(\Omega))^*)}  \ln^{1/2}\left  ( 1 +  \| \t \f u \|_{ (\f W^{2,p}_{0,\sigma}(\Omega))^*)} \right ) \de t \leq c \quad \text{for } p > d \,,
\end{align*} 
since $W^{2,p}(\Omega)$ is embedded into $\mathcal{C}^1(\overline  \Omega )$. 
Concerning the convection terms, we observe 
\begin{multline}
\| \f u \varphi \|_{L^2 (0,T; L^{d/(d-1)} (\Omega))} + \| \f u f'_\delta (\theta)  \|_{L^2 (0,T; L^{d/(d-1)} (\Omega))} \\
\leq  \| \f u \|_{L^\infty(0,T; \f H )} \left ( c\| \varphi \|_{L^\infty(0,T;L^{2d/d-2}(\Omega))} + \| f'_\delta(\theta)  \|_{L^2(0,T;L^{2d/d-2}(\Omega))} \right ) \,,\label{convterms}
\end{multline}
where the right-hand side is bounded due to~\eqref{aprienergy},~\eqref{nablafprime}, the embedding $H^1 \hookrightarrow L^{2d/(d-2)}$, and Poincar\'e's inequality.
By comparison in~\eqref{eq2}$_1$, one may find
\begin{align}
\int_0^T \| \t \varphi \|_{
(\f W^{1,p}(\Omega))^*} \ln^{1/2} \left ( 1+\| \t \varphi \|_{
(\f W^{1,p}(\Omega))^*} \right )\de t  \leq c \quad \text{for } p > d \,, \label{apri:phitime}
\end{align}
 since $W^{1,p}(\Omega)$ is embedded into $\mathcal{C}(\overline  \Omega )$. 
Similar, comparison in the entropy balance provides
\begin{align*}
\| \t f_\delta'(\theta) \|_{\mathcal{M}([0,T]; (\f W^{1,p}(\Omega))^*)} \leq c \quad \text{for } p > d  \,.
\end{align*}
\textbf{Additional regularity.}
 Finally, we observe that under the condition~\eqref{addreg}, the defect measure $m$ vanishes.
 Indeed, in this case, we infer from the entropy inequality~\eqref{aprientro} that $ \| \theta ^{\beta/2} \|_{L^2(W^{1,2})}\leq c$, which implies due to an embedding that  $\| \theta^\beta \| _{L^1 (L^{p})} \leq c$ for $p<\infty$ for $d =2$ and $p=d/(d-2)$ for $d\geq 3$. 
 From the energy estimate~\eqref{aprienergy}, we infer $\| \theta ^\delta \|_{L^\infty(L^1)}\leq c$. Interpolating between these two spaces under the assumption~\eqref{addreg} implies that
 \begin{align*}
\| \theta  \|_{L^q(\Omega\times (0,T))} \leq c \quad \text{for } q > \frac{d}{2}\,.
 \end{align*}
 Together with the entropy bound~\eqref{aprientro} this implies improved bounds on $\f u$ and $\mu$, \textit{i.e.}, 
 \begin{align*}
 \| \nabla \f u  \|_{L^p(\Omega \times (0,T))}+
 \| \nabla \mu  \|_{L^p(\Omega \times (0,T))} \leq \int_0^T\int_\Omega \left (\frac{\nu(\theta )|(\nabla \f u)_{\sym} |^2 }{\theta } + \frac{| \nabla \mu |^2}{\theta   }\right )  \de \f x \de t + \| \theta  \|_{L^q(\Omega\times (0,T))} \leq c
 \end{align*}
for $p> 2d/(d+2)$, where Korn's inequality is used~\cite[Thm.~10.15]{singular}.
An embedding together with Poincar\'e's inequality implies that 
\begin{align*}
\| \f u  \|_{L^r(\Omega \times (0,T))} + \| \mu  \|_{L^r(\Omega \times (0,T))}\quad \text{for }r>2 \,.
\end{align*} 
This in turn already implies that a hypothetical approximate sequence  $\{ \f u_n \otimes \f u_n\} $ is relative weakly compact in $L^1(\Omega \times (0,T); \R^{d\times d})$. 

Since $F$ is $\lambda$-convex, we may define $ G:= F+\lambda I$ such that $G$ is convex. 
To infer additional regularity for $\nabla \varphi $, we test equation~\eqref{eq2}$_2$ by $|G'(\varphi )|^{s-2}G'(\varphi )$ to infer that
\begin{multline*}
\int_0^T \int_\Omega (s-1) G''(\varphi  ) | \nabla \varphi |^2 | G'(\varphi )|^{s-2} + | G'(\varphi )|^s \de \f x \de t \\ 
= \int_0^T \int_\Omega \left ( \theta  + \mu  + \lambda \varphi   \right ) | G'(\varphi )|^{s-2} G'( \varphi  ) \de \f x \de t\qquad\qquad\qquad\qquad\qquad\qquad\quad \\
\leq  \frac{s-1}{s} \left \| G'(\varphi) \right \| _{L^s(\Omega \times (0,T))} ^s + \frac{1}{s} \left ( \| \theta   \| _{L^s(\Omega\times (0,T))}^s  + \| \mu  \| _{L^s(\Omega\times (0,T))}^s +\lambda^s\| \varphi    \| _{L^s(\Omega\times (0,T))}^s \right ) \,. 
\end{multline*}
For $s=\min\{r,q\}$, the first term on the right-hand side of the previous inequality may be absorbed into the left-hand side and the  other terms are bounded. 
Such that by comparison in~\eqref{eq2}$_2$, we observe $ \| \Delta \varphi  \|_{L^s (\Omega \times (0,T))} \leq c $ and by an embedding
we find 
\begin{align*}
\| \nabla \varphi   \|_{ L^p(\Omega \times (0,T))} \leq c \quad \text{for } p > \min\left \{ d, \frac{2d}{d-2}\right \} >2\,.
\end{align*}
This in turn already implies that a hypothetical approximate  sequence  $\{ \nabla\varphi_n   \otimes \nabla \varphi _n\} $ is relative weakly compact in $L^1(\Omega \times (0,T); \R^{d\times d})$. 
Note that in the case $d=3$ the relatively weakly compactness property of $\{ \nabla\varphi_n   \otimes \nabla \varphi _n\} $ in $L^1(\Omega \times (0,T); \R^{d\times d})$ could already be achieved by choosing $ \delta \geq 6/5-6/5\beta$ instead of~\eqref{addreg}.

 \begin{remark}
 To infer the estimates on the time derivatives of the solutions to the discrete system~\eqref{dis} rigorously, some stability properties of the $L^2(\Omega)$-Projection onto the Galerkin spaces are needed. 
 \end{remark}
 \subsection{Weak sequential compactness\label{sec:sequential}}
 Considering a hypothetical approximate sequence $\{ ( \f u_n , \theta _n , \varphi_n, \mu_n )\}$, we are going to prove the weak sequential compactness of the measure-valued formulation~\eqref{def:weak}. As an approximate sequence the solutions to~\eqref{dis} could be chosen. 
 Collecting the bounds from the previous section, we observe
\begin{align}
\f u_n  &
 \stackrel{*}
 {\rightharpoonup} \f u \quad \text{in } L^\infty ( 0,T ; \f H )\cap L^1(0,T;\f W^{1,1}(\Omega))\label{convun}\,,
\\\t \f u_n  &
 {\rightharpoonup} \t \f u \quad \text{in }L^{1} (0,T; (\f W^{2,p}_{0,\sigma} (\Omega))^*)\,,
\\
 \varphi_n  &
 \stackrel{*}
 {\rightharpoonup} \varphi  \quad \text{in } L^\infty(0,T; \f H^1(\Omega)) \cap L^1(0,T; W^{2,1}(\Omega))\,,
\\ \t \varphi_n  &
 {\rightharpoonup} \t \varphi  \quad \text{in }  L^1(0,T; (\f W^{1,p}(\Omega))^*)\,,\label{conv:phitime}
\\
\hat{\kappa}(\theta_n) & {\rightharpoonup}  \eta \quad \text{in } L^2 ( 0,T; H^1(\Omega))
  \,,\label{w:log}
\\
f' (\theta_n )   &\stackrel{*}{\rightharpoonup}  \zeta \quad \text{in } L^2 ( 0,T; H^1(\Omega))\cap 
L^\infty(0,T;\mathcal{M}(\ov\Omega)) \,,
\\
\t f' (\theta_n)  &\stackrel{*}{\rightharpoonup} \t f' (\theta) \quad \text{in } \mathcal{M}([0,T];(\f W^{1,p}(\Omega))^*) \text{ for }p>d\,,
\\
\mu_n & {\rightharpoonup} \mu \quad \text{in } L^1(0,T;W^{1,1}(\Omega)) \,.\label{conv:mu}
\intertext{The Lions--Aubin lemma (see~\cite[Cor.~7.9]{roubicek}) grants that }
\f u_n  & \ra \f u  \quad \text{in } L^p(0,T;L^q
(\Omega)) \text{ for all }p\in [1,\infty)\text{ and }q \in [1,2)\,,\label{strongu}
\\
\varphi_n  & \ra \varphi  \quad \text{in }L^p(0,T;L^r
(\Omega)) \text{ for all }p\in [1,\infty) \text{ and }r \in [1,2d/(d-2)) \,,\label{strongphi}
\\
\nabla \varphi_n  & \ra \nabla \varphi  \quad \text{in }L^1(0,T;L^q
(\Omega)) \text{ for all }p\in [1,\infty) \text{ and }q \in [1,2)\,,\label{strongnablaphi}
\\
  f'(\theta_n) &\ra  \xi    \quad \text{in } L^2(0,T;L^2(\Omega) )\,. \label{psithetastrong}
\end{align}

Due to~\eqref{psithetastrong}, we can extract a subsequence that converges a.e. in $\Omega \times (0,T)$, \textit{i.e.}, $ f'_\delta(\theta_n) \ra \xi $ a.e.~in $\Omega \times (0,T)$. Since $f'_\delta$ is a bijective function, we define $\theta = (f'_\delta)^{-1} (\xi )$  and observe that $\theta_n\ra \theta $ a.e.~in $\Omega \times (0,T)$. 
Vitali's theorem  together with~\eqref{apri:theta} implies
\begin{align}
  \theta_n &\ra \theta   \quad \text{in } L^1(0,T;L^1(\Omega) )\,. \label{thetastrong}
\end{align}
The continuity of $f'_\delta$ and $Q_\delta$ implies that $ f'_\delta ( \theta_n) \ra f'_\delta (\theta) $ and $ Q_\delta ( \theta _n) \ra Q_\delta ( \theta) $ a.e.~in $\Omega \times (0,T)$. 

Since the time-derivative of the sequences $\{ \f u_n\}$ and $\{\varphi_n\}$  converge weakly in $L^1(0,T;(W^{2,p}(\Omega))^*)$ and $L^1(0,T;(W^{1,p}(\Omega))^*)$ for $p>d$, respectively, we deduce that
$\f u_n \ra \f u $ in $\mathcal{C}_w ([0,T]; (W^{2,p}(\Omega))^*)$ and $\varphi \ra \varphi $ in $\mathcal{C}_w ([0,T]; (W^{1,p}(\Omega))^*)$, which implies by a standard lemma (see~\cite[page~297]{magnes}) that
\begin{align*}
\f u_n  & \ra \f u  \quad \text{in } \mathcal{C}_w([0,T];\f L^2(\Omega)) \quad\text{and} \quad \varphi_n    \ra \varphi  \quad \text{in } \mathcal{C}_w([0,T];H^1(\Omega)) \,.
\end{align*}
The energy bounds allow to deduce the existence of a measure $ m \in L^\infty(0,T; \mathcal{M}(\ov \Omega; \R^{d\times d})$ (see~\cite{alibert,RoubicekMeasure}, or compare~\cite{meas}) such  that
\begin{align}
\f u _n \otimes \f u_n +\varepsilon \nabla \varphi_n \otimes \nabla \varphi _n  &\stackrel{*}{\rightharpoonup} \f u \otimes \f u + \varepsilon\nabla \varphi \otimes \nabla \varphi + m \quad \text{in }  L^\infty (0,T; \mathcal{M}(\ov \Omega; \R^{d\times d})\,.\label{convmeasure}
\end{align}
By the lower semi-continuity of weak convergence, we may observe that $ m$ is a semi-positive matrix, \textit{i.e.}, for any $ \f a \in \C^\infty _c (\ov \Omega \times [0,T]; \R^d)$, we find
\begin{align*}
\left \langle m ; \f a \otimes \f a \right \rangle ={}& \lim_{n \ra \infty}\int_\Omega \left ( \left ( \f u_n \otimes \f u_n + \varepsilon \nabla \varphi_n \otimes \nabla \varphi_n \right )  - \left (  \f u \otimes \f u + \varepsilon \nabla \varphi \otimes \nabla \varphi \right )\right ) : \left ( \f a \otimes \f a \right )\de \f x \\  \geq {}& \liminf_{n \ra \infty} \int_\Omega 
\left ( \left ( \f u _n \cdot \f a \right )^2  + \varepsilon \left ( \nabla \varphi_n \cdot \f a \right ) ^2 \right ) - \left ( \left ( \f u  \cdot \f a \right )^2  + \varepsilon \left ( \nabla \varphi \cdot \f a \right ) ^2 \right )\de \f x \geq 0 \,.
\end{align*}
 The pointwise strong convergence implied by~\eqref{strongu} and~\eqref{strongnablaphi} allow to deduce that $m$ is indeed a defect measure, \textit{i.e.}, the Lebesgue-part is zero. 
 %
%
 The estimate~\eqref{convterms} implies that the convective terms in~\eqref{phaseeq} and~\eqref{entropy} are relatively weakly compact in $L^1(\Omega \times (0,T))$. 
 
 With these different convergences at hand, it is a standard matter to pass to the limit in the formulation~\eqref{weakmomentum}. 
The convergences also allow to pass to the limit in the formulation~\eqref{mueq}. Indeed, Hypothesis~\ref{hypo} allows to infer that $ \{ F'( \varphi_n)\}$ is relatively weakly compact, by the a.e.~convergence of $ \{ \varphi_n\}$, the continuity of $F'$, and Vitali's theorem, we find
\begin{align*}
F'(\varphi_n ) \ra F'(\varphi) \quad  \text{in } L^1 ( 0,T ; L^1(\Omega))\,,
\end{align*}
which allows together with~\eqref{thetastrong} to pass to the limit in~\eqref{mueq}. 
From~\eqref{convun} and~\eqref{strongeqphi}, we may deduce
\begin{align*}
\f u_n \varphi_n \rightharpoonup  \f u \varphi \quad \text{in } L^{d/(d-1)} ( \Omega \times (0,T))\,.
\end{align*}
This, together with~\eqref{conv:phitime} and~\eqref{conv:mu} allow to pass to the limit in~\eqref{phaseeq}. 
%
%
%
The convergence in the energy inequality is observed by~\eqref{convmeasure} multiplied by the identity, the a.e.-convergence of $\{\varphi_n\} $ and $\{ \theta_n\}$, the continuity and coercivity of $F$ and $Q_\delta $, as well as Fatou's lemma~\cite{fatou}. 

Passing to the limit on the right-hand side of~\eqref{entropy} is straightforward, where estimate~\eqref{convterms} guarantees that the convective term is relatively weakly compact in $L^1(\Omega \times (0,T))$. For the left-hand side, we observe that the dissipative terms under the time-intergral converge due to the lower semi-continuity of convex functions (see Ioffe~\cite{ioffe} or~ Thm.~2 in \cite{giulio}). For the entropic part, which is the first term in~\eqref{entropy}, the convergence for $\varphi$ is obvious, since $\varphi_n \ra \varphi$ in $\mathcal{C}_w([0,T];H^1(\Omega))$. 
For $\delta>0$,  $ \{ f'_\delta(\theta_n)\}$ is relatively weakly compact in $L^1(\Omega\times (0,T))$ due to the energy estimate~\eqref{aprienergy}, which together with the a.e.-point-wise convergence of $ \{ \theta_n \}$ 
and Vitali's theorem implies the  convergence of $\left \{\int_\Omega f'_\delta (\theta_n ) \de \f x \right \} $ to $ \int_\Omega f'_\delta(\theta) \de \f x $  for a.e.~$t\in (0,T)$. 
In the case $\delta=0$, we may argue similar on the set $\theta \geq 1$. For $\theta<1$ we may argue by the positivity of $- \ln \theta $  and Fatou's lemma see~\cite{indu} or~\cite{larosc}. 
 
 In case that~\eqref{addreg} holds, the~\textbf{additional regularity} holds, the sequences $ \{ \f u _n \otimes \f u_n +\varepsilon \nabla \varphi_n \otimes \nabla \varphi _n  \} $ is relatively weakly compact in $L^1(\Omega \times (0,T))$ such that $m$ is vanishing~\cite[Lem.~3.2.14]{RoubicekMeasure}. 
 

\section{Relative energy inequality\label{sec:rel}}
This section is devoted to the proof of the relative energy inequality. For convenience we set $\varepsilon=1$ in this section. All calculations may be adapted to varying $\varepsilon$ easily. 

\begin{proposition}\label{prop:rel}
Let $( \f u, \theta , \varphi )$ together with $\mu \in L^2(0,T;W^{1,1}(\Omega))$ and $m\in L^\infty(0,T;\mathcal{M}( \ov \Omega ; \R^{d\times d} _{\sym,+}))$ be a measure-valued solution according to Definition~\ref{def:weak} and let $(\tu,  \tet,\tp)\in \mathbb{Y}$ be a regular weak  solution according to Definition~\ref{def:weak}. 
Then the relative energy inequality
\begin{multline}
\mathcal{R}(\f u, \theta , \varphi  |\tu,  \tet,\tp) (t) + \frac{1}{2}\langle m , I \rangle + \int_0^t \mathcal{W} (\f u, \theta , \varphi  |\tu,  \tet,\tp)  e^{\int_s^t  \mathcal{K} (\tu,  \tet,\tp) \de \tau } \de s 
\\
\leq 
\mathcal{R}(\f u, \theta , \varphi  |\tu,  \tet,\tp) (0) e^{\int_0^t  \mathcal{K} (\tu,  \tet,\tp) \de s }
  \,,\label{relenenergyweak}
\end{multline}
holds for a.e.~$t\in (0,T)$ and thus the assertion.

\end{proposition}

\subsection{Relative energy} 

The following calculation hold for a.e.~$t\in (0,T)$. 
Regrouping of the appearing terms in~\eqref{relen} gives
\begin{subequations}\label{calcrelen}
\begin{align}
\mathcal{R}(\f u , \theta , \varphi  | \tu,\tet,\tp) + \frac{1}{2}\langle  m , I \rangle={}& \int_{\Omega} \left (\frac{1}{2}| \f u |^2 +  \frac{1}{2}| \nabla \varphi|^2 + F (\varphi) +Q_\delta(\theta) \notag
 \right )\de \f x + \frac{1}{2}\langle  m , I \rangle \\ &  +\int_{\Omega} \left ( \frac{1}{2}|\tu|^2 + \frac{1}{2}| \nabla \tp|^2 + F (\tp) +Q_\delta(\tet) 
  \right )\de \f  x\notag \\
&-  \int_{\Omega} \left ( \f u \cdot \tu - \tet ( f'_\delta( \theta) - \varphi )  \right ) \de\f x+ \frac{M}{2} \| \varphi - \tp \|_{W^{-1,\infty}(\Omega)} ^2 \notag
 \\ 
& -  \int_{\Omega} \left (  \nabla \tp \cdot  \nabla \varphi  + (  F' ( \tp ) (  \varphi-\tp  )+ 2 F(\tp) )  \right ) \de \f x \label{phaseterms}
\\
&- \int_{\Omega} \left ( f_\delta(\tet) + Q_\delta(\tet) - \tet \varphi - {\lambda}| \varphi - \tp |^2    \right ) \de \f x \,. \label{mixedterms}
\end{align}
\end{subequations}
First, we observe by the energy inequality~\eqref{energyin} for the weak solution and the energy equality (\eqref{energyin} with equality) for the strong solution that
\begin{multline}
 \int_{\Omega} \left (\frac{1}{2}| \f u(t)|^2+  \frac{1}{2}| \nabla \varphi(t)|^2 + F (\varphi(t)) +Q_\delta(\theta(t)) 
    \right ) \de \f x
     + \frac{1}{2}\langle  m , I \rangle \\+\int_{\Omega} \left ( \frac{1}{2}| \tu(t)|^2+ \frac{1}{2}| \nabla \tp(t)|^2 + F (\tp(t)) + Q_\delta(\tet(t)) 
      \right ) \de \f x \\
\leq   \int_{\Omega} \left (\frac{1}{2}| \f u_0 |^2 +  \frac{1}{2}| \nabla \varphi_0|^2 + F (\varphi_0) + Q_\delta (\theta _0)\right ) + \int_{\Omega} \left (\frac{1}{2} | \tu_0|^2+ \frac{1}{2}| \nabla \tp_0|^2 + F (\tp_0) + Q_\delta (\tet_0 )\right )\,. \label{energyestimates}
\end{multline}
%
Testing the weak form of the momentum balance with $ \tu$, \textit{i.e.,} choosing $\f \xi = \tu$ in~\eqref{weakmomentum}, we find
\begin{multline*}
-  \int_\Omega \f u \cdot \tu \de \f x \Big |_0^t + \int_0^t \int_\Omega \f u \t \tu + ( \f u \otimes \f u ) : \nabla \tu  \de \f x \de s\\ = \int_0^t \int_\Omega \nu (\theta ) ( \nabla \f u )_{\sym} : (\nabla \tu )_{\sym} - (\nabla \varphi \otimes \nabla \varphi ) : (\nabla \tu )_{\sym} \de \f x  - \int_{\overline\Omega}( \nabla \tu )_{\sym} : m(\de \f x)  \de s\,.
\end{multline*}
Inserting additionally the strong formulation of the momentum balance~\eqref{eq1} for $\tu$ tested with $\f u$, we observe
\begin{align}
\begin{split}
-  \int_\Omega \f u \cdot \tu \de \f x \Big |_0^t -{}& \int_0^t \int_\Omega( \nu (\theta )+ \nu(\tet)) ( \nabla \f u )_{\sym} : (\nabla \tu )_{\sym}   \de \f x \de t\\ ={}&-  \int_0^t \int_\Omega ( \f u \otimes \f u ) : \nabla \tu + ( \tu \otimes \tu) : \nabla \f u   \de \f x + \int_{\overline\Omega}( \nabla \tu )_{\sym} : m(\de \f x)  \de s
\\
&-\int_0^t \int_\Omega (\nabla \varphi \otimes \nabla \varphi ) : (\nabla \tu )_{\sym} + ( \nabla \tp \otimes \nabla \tp) : ( \nabla \f u ) _{ \sym}   \de \f x \de s
\,.\end{split}
\label{momentumtest}
\end{align}
Choosing $ \vartheta = \tet$ in~\eqref{entropy}, one may observe
\begin{multline}
  \int_{\Omega} \tet  (f'_\delta  (  \theta) - \varphi) \de \f x  \Big|_0^t + \int_0^t \int_{\Omega} \tet  \left ( \kappa (\theta)  | \nabla  \log \theta|^2+ \frac{\nu(\theta) | (\nabla \f u_{\sym}|  ^2}{\theta }+ \frac{|\nabla  \mu |^2 }{\theta}  \right ) \de \f x  \de s\\
  \leq \int_0^t \int_\Omega\left (  \kappa (\theta)    \nabla \log \theta \cdot \nabla \tet + ( \t \tet  + (\f u \cdot \nabla ) \tet) (f_\delta' ( \theta) - \varphi )\right ) \de \f x  \de s \,.\label{entrotest}
\end{multline}
Similar, we find by testing the energy balance~\eqref{eq1} for the strong solution with $\theta$ that 
\begin{multline}
 \int_0^t \int_{\Omega} \theta  \left ( \kappa (\tet)  | \nabla  \log \tet|^2+ \di ( \kappa(\tet) \nabla \log \tet) + \frac{\nu(\tet) |( \nabla \tu)_{\sym} |  ^2}{\tet }+ \frac{|\nabla  \tilde{\mu} |^2 }{\tet}  \right ) \de \f x  \de s\\
  = -\int_0^t \int_\Omega  \theta( \t   + (\tu \cdot \nabla ) ) ( f_\delta' ( \tet) - \tp ) \de \f x  \de s \,.\label{entrotestzwei}
\end{multline}
For the terms in line~\eqref{phaseterms}, we find with the fundamental theorem of calculus that
\begin{multline*}
- \int_{\Omega} \left (   \nabla \tp \cdot  \nabla \varphi  + (  F' ( \tp ) (  \varphi-\tp  )+ 2 F(\tp) ) \right ) \de\f x \Big|_0^t
\\
= - \int_0^T\int_{\Omega}  \nabla \t \tp \cdot  \nabla \varphi \de \f x - \langle  \t \varphi, \Delta \tp   \rangle\de s\qquad\qquad
\\- \int_0^T \langle  \t \varphi-\t \tp ,  F' ( \tp )\rangle  +\int_{\Omega} \left ( F'' (\tp)\t \tp (\varphi - \tp) + 2 F'(\tp) \t \tp  \right ) \de\f x \de s \,,
\end{multline*}
where this formula first only holds for more regular functions, but can be extended by  density arguments.
Choosing $\phi =- \tilde{\mu} $~\eqref{phaseeq},  $ \zeta=\t \tp$ in~\eqref{mueq}, and adding both equations provides
\begin{align*}
-\int_0^t\int_\Omega & \nabla \varphi \cdot \nabla \t \tp +{} \nabla \mu \cdot \nabla \tilde{\mu }\de \f x \de t 
\\={} &
\int_0^t \langle  \t \varphi, \tmu \rangle - \int_\Omega ( \f u  \varphi )\cdot \nabla \tilde{\mu} + \theta \t \tp + \mu \t \tp  - F'(\varphi) \t \tp  \de \f x \de s \\
 ={}& \int_0^t  \langle  \t \varphi, \tmu \rangle - \int_\Omega   ( \f u  \varphi)\cdot \nabla   F'(\tp)  -  ( \f u  \varphi)\cdot \nabla\Delta \tp - ( \f u  \varphi)\cdot \nabla \tet  + \theta \t \tp + \mu \t \tp - F'(\varphi) \t \tp  \de \f x \de s \,.
\end{align*}
Similar calculations for the strong solution, \textit{i.e.}, testing~\eqref{eq2}$_1$ with $\mu$ and inserting equation~\eqref{eq2}$_2$ twice, we find
\begin{multline}
\int_0^t   \langle     \t \varphi,\Delta \tp\rangle  - \langle \t \varphi, F'(\tp) \rangle  -\int_\Omega \nabla \mu \cdot \nabla \tilde{\mu }\de \f x \de s\\
 ={} \int_0^t \int_\Omega  \t \tp \mu +   \nabla(( \tu \cdot \nabla ) \tp)\cdot  \nabla \varphi  +( \tu \cdot \nabla ) \tp F'(\varphi) - ( \tu \cdot \nabla ) \tp \theta \de \f x-\langle  \t \varphi , \tet +  \tmu\rangle    \de s \,.\label{strongeqphi}
\end{multline}
Combining the last three equations, we arrive at
\begin{align}
\begin{split}
- \int_{\Omega}& \left (   \nabla \tp \cdot  \nabla \varphi  + (  F' ( \tp ) (  \varphi-\tp  )+ 2 F(\tp) ) \right ) \de\f x \Big|_0^t  - 2 \int_0^t \int_\Omega \nabla \mu \cdot \nabla \tilde{\mu} \de \f x \de s\\
={}&- \int_0^t \int_\Omega(   \f u  \varphi )\cdot \nabla F'(\tp)-  ( \f u  \varphi )\cdot \nabla\Delta \tp  -  ( \f u  \varphi )\cdot \nabla\tet+ \theta \t \tp     \de \f x \de s
\\
&+  \int_0^t \int_\Omega \nabla(( \tu \cdot \nabla ) \tp)\cdot  \nabla \varphi +( \tu \cdot \nabla ) \tp F'(\varphi) - ( \tu \cdot \nabla ) \tp \theta\de \f x - \langle  \t \varphi , \tet \rangle    \de s
\\
&+ \int_0^t \int_\Omega\t \tp \left (  F'(\varphi) - F'(\tp) - F'' (\tp) (\varphi - \tp)\right ) \de \f x \de s \,.
\end{split}\label{phaseinserted}
\end{align}
Applying now the fundamental theorem of calculus to the terms in line~\eqref{mixedterms}, we find
\begin{align}
\begin{split}
- \int_{\Omega} &\left (
2 f_\delta (\tet) - \tet f_\delta' (\tet) - \tet \varphi   \right ) \de \f x\Big|_0^t \\
={}& -\int_0^t \int_\Omega 2f'_\delta (\tet) \t \tet - \t \tet f'_\delta (\tet) - \tet f''_\delta (\tet) \t\tet - \t\tet \varphi  \de \f x  -\langle   \t \varphi,\tet \rangle   \de s 
\\={}& -\int_0^t \int_\Omega   \t \tet( f'_\delta (\tet) - \tet f''_\delta (\tet) ) - \t\tet \varphi \de \f x  -\langle   \t \varphi,\tet \rangle  \de s  \,.
\end{split}\label{integrated}
\end{align} 
Inserting now~\eqref{energyestimates},~\eqref{momentumtest},~\eqref{entrotest}~\eqref{entrotestzwei},~\eqref{phaseinserted}, and~\eqref{integrated} into~\eqref{calcrelen}, we observe
\begin{align}
\begin{split}
\mathcal{R}(\f u(t), \theta (t), \varphi (t) &|{}\tu(t),  \tet(t),\tp(t))  
 \\+ \int_0^t \int_\Omega& \left (  
\tet \frac{\nu(\theta)|(\nabla \f u)_{\sym} |^2 }{\theta} + \theta \frac{\nu(\tet) |( \nabla \tu)_{\sym}|^2}{\tet} - ( \nu(\theta ) + \nu(\tet) ) (\nabla \f u)_{\sym} : (\nabla \tu)_{\sym} 
    \right ) \de \f x  \de s \\
    + \int_0^t \int_\Omega& \left (  \tet\frac{| \nabla \mu|^2}{\theta}  + \theta \frac{| \nabla \tilde{\mu }|^2}{\tet} - 2 \nabla \mu \cdot \nabla \tilde{\mu}   \right )\de \f x  \de s \\
+ \int_0^t \int_\Omega& \left (   \tet  \kappa(\theta)    {| \nabla \log\theta|^2 } -    \kappa(\theta) {\nabla \log\theta }\cdot \nabla \tet +  \theta  \kappa(\tet)  {| \nabla\log \tet|^2} +    {\di (\kappa(\tet) \log \tet ) }   \theta   \right )\de \f x  \de s \\
\leq{}& \mathcal{R }(\f u_0 , \theta_0 , \varphi _0 | \tu_0 , \tet_0,\tp_0 ) + \int_0^t \int_\Omega\t \tp \left (  F'(\varphi) - F'(\tp) - F'' (\tp) (\varphi - \tp)\right ) \de \f x \de s\\
& -  \int_0^t \int_\Omega ( \f u \otimes \f u ) : \nabla \tu + ( \tu \otimes \tu) : \nabla \f u \de \f x \  + \int_{\overline\Omega}( \nabla \tu )_{\sym} : m_t(\de \f x)  \de s   
\\
&-\int_0^t \int_\Omega (\nabla \varphi \otimes \nabla \varphi ) : (\nabla \tu )_{\sym} + ( \nabla \tp \otimes \nabla \tp) : ( \nabla \f u ) _{ \sym}   \de \f x \de s\\
&+\int_0^t \int_\Omega  ( \t \tet  + (\f u \cdot \nabla ) \tet) (f_\delta'( \theta) - \varphi )  -\left ( \theta( \t   + (\tu \cdot \nabla ) ) (f'_\delta ( \tet) - \tp )\right ) \de \f x  \de s \\
&-  \int_0^t \int_\Omega  ( \f u  \varphi )\cdot \nabla F'(\tp) - ( \f u  \varphi )\cdot \nabla \Delta \tp  -( \f u  \varphi )\cdot \nabla\tet  +\theta \t \tp \de \f x \de s
\\
&+  \int_0^t \int_\Omega  \nabla(( \tu \cdot \nabla ) \tp)\cdot  \nabla \varphi +( \tu \cdot \nabla ) \tp F'(\varphi) - ( \tu \cdot \nabla ) \tp \theta \de \f x -\langle  \t \varphi , \tet \rangle   \de s 
\\
& -\int_0^t \int_\Omega   \t \tet( f'_\delta (\tet) - \tet f''_\delta (\tet) ) - \t\tet \varphi \de \f x - \langle  \t \varphi , \tet \rangle \de s 
+ \frac{M}{2} \| \varphi - \tp \|_{W^{-1,\infty}(\Omega)} ^2 
\,.
 \end{split}\label{firstest}
\end{align}
For the term due to the convection in the fluid, we find 
\begin{align*}
-\int_\Omega& ( \f u \otimes \f u ) : \nabla \tu + ( \tu \otimes \tu) : \nabla \f u   \de \f x
\\
={}& -\int_\Omega (  (\f u-\tu)\otimes \f u ) : \nabla \tu + (  (\tu -\f u)\otimes \tu ) : \nabla \f u   \de \f x  \\
={}&- \int_\Omega ( (\f u-\tu) \otimes (\f u-\tu) ) : \nabla \tu + (  (\tu -\f u)\otimes \tu ) :(  \nabla \f u - \nabla \tu)    \de \f x   \\
={}& -\int_\Omega ( (\f u-\tu) \otimes (\f u-\tu) ) : (\nabla \tu )_{\sym}    \de \f x \\
\leq{}&\left \| (\nabla \tu )_{\sym,-} \right \|_{L^\infty(\Omega)}  \| \f u -\tu\|_{L^2(\Omega)}^2 \,,
\end{align*}
where $ ( \nabla \tu )_{\sym,-}$ denotes the negative part of this symmetric matrix (see Section~\ref{sec:not}).
The first equality follows from the fact that $\f u$ and $\tu$ are divergence free, such that
$ \f u \cdot \nabla | \tu|^2 $ and $\tu \cdot \nabla | \f u|^2$ integrated over $\Omega$ vanish. 
The second equality  is just a rearrangement and the third follows again from the fact that $\tu$ is a solenoidal vector field. 
Concerning the defect measure $m$, we may estimate
\begin{align*}
 \int_{\overline\Omega}( \nabla \tu )_{\sym} : m(\de \f x)  \leq \| ( \nabla \tu )_{\sym,-} \|_{L^\infty(\Omega)} \left \langle m ,1 \right \rangle \,,
\end{align*}
 since the positive part may be estimated by zero due to the semi-positiveness of the matrix $m (\de \f x)$. 

For the coupling terms of the fluid and the phase-field equation, we observe
\begin{align*}
\int_\Omega& \nabla\left (  ( \tu\cdot \nabla )\tp\right ) \cdot \nabla \varphi +  ( \f u\varphi ) \cdot \nabla  \Delta \tp \de \f x \\
-&\int_\Omega ( \nabla \varphi \otimes \nabla \varphi ): ( \nabla \tu)_{\sym} +( \nabla \tp \otimes \nabla \tp):( \nabla \f u )_{\sym}\de \f x \\
= {}& \int_\Omega  \nabla \tu : \left ( \nabla \tp \otimes \nabla \varphi \right ) + ( \tu \cdot \nabla ) \nabla \tp \cdot \nabla \varphi +  ( \f u  \varphi )\cdot \nabla  \Delta \tp \de \f x 
\\ & - \int_\Omega \nabla \tu : \left ( \nabla \varphi \otimes \nabla \varphi \right ) - ( \f u \cdot \nabla ) \tp \Delta \tp \de \f x 
\\={}& \int_\Omega \nabla \tu : \left ( \left ( \nabla \tp - \nabla \varphi \right ) \otimes \nabla \varphi \right ) + ( \tu \cdot \nabla ) \nabla \tp \cdot \nabla \varphi  + ( \f u ( \varphi - \tp) ) \cdot \nabla \Delta \tp \de \f x 
\\={}& -\int_\Omega \nabla \tu : \left ( \left ( \nabla \varphi - \nabla \tp \right ) \otimes \left ( \nabla \varphi - \nabla \tp \right )\right ) +\left  (\left ( \f u - \tu\right  ) \left  ( \varphi - \tp \right )\right )\cdot \nabla \Delta \tp \de \f x 
\\& + \int_\Omega \nabla \tu : \left ( \left ( \nabla \tp - \nabla \varphi\right ) \otimes \nabla \tp \right ) + ( \tu \cdot \nabla ) \nabla \tp \cdot \left ( \nabla \varphi - \nabla \tp  \right ) + ( \tu  ( \varphi - \tp ))\cdot \nabla ) \Delta \tp \de \f x 
%
%
%
\,.
\end{align*}
The first equality in the above equality chain follows from the product rule, a rearrangement, and an integration-by-parts on the  last term. The second and third equality are just rearrangements, while once using that   $ \tu$ is divergence free, such that the integral over $ ( \tu \cdot \nabla ) | \nabla \tp|^2 $  vanishs. 
Finally, we again observe by an integration-by-parts rule and  since $\tu$ is a solenoidal vector field that the last line of the above equation vanishes. 
Note the the integration-by-parts rule for the last step initially only holds for more regular function, but may  be  extended by density arguments. 
%
%
%
%
The terms due to the nonconvex potential in~\eqref{firstest}, gives
\begin{align*}
-\int_\Omega & ( \f u  \varphi )\cdot \nabla  F'(\tp) - (\tu\cdot \nabla ) \tp F'(\varphi) \de \f x \\
={}& \int_\Omega ( \f u  (\tp -\varphi) )\cdot \nabla  F'(\tp) + (\tu\cdot \nabla )\tp \left (  F'(\varphi) - F'( \tp) - F''(\tp) ( \varphi - \tp) \right )  + ( \tu \cdot \nabla ) \tp F''(\tp) ( \varphi - \tp) \de \f x \\
={}& \int_\Omega ( (\f u-\tu)  (\tp - \varphi ))\cdot \nabla   F'(\tp) + (\tu\cdot \nabla )\tp \left (  F'(\varphi) - F'( \tp) - F''(\tp) ( \varphi - \tp) \right )  \de \f x 
\\&- \int_\Omega ( \tu  ( \varphi - \tp))\cdot \nabla  F'(\tp) - ( \tu \cdot \nabla ) \tp F''(\tp) ( \varphi - \tp) \de \f x 
 \,.
\end{align*}
The first equality is valid since $\f u$ and  $\tu$ are solenoidal functions. Indeed by $ ( \f u \cdot \nabla) \tp F' ( \tp )= ( \f u \cdot \nabla) F(\tp) $, we observe that the integral over this term vanishes. Similar, this holds for the term $ ( \tu\cdot \nabla ) F(\tp)= ( \tu \cdot \nabla ) \tp F'(\tp) $. 
The second equality is just a rearrangement.
The last line of the above equation vanishes again, which may be inferred from the fact that $\tu$ is a solenoidal vector field. 
Together, we  estimate by  {H\"older}'s inequality,
\begin{align*}
\int_\Omega& \nabla\left (  ( \tu\cdot \nabla )\tp\right ) \cdot \nabla \varphi + ( \f u  \varphi )\cdot \nabla \Delta \tp \de \f x
-\int_\Omega ( \nabla \varphi \otimes \nabla \varphi ): ( \nabla \tu)_{\sym} +( \nabla \tp \otimes \nabla \tp):( \nabla \f u )_{\sym}\de \f x \\
-\int_\Omega & ( \f u \varphi )\cdot \nabla   F'(\tp) + (\tu\cdot \nabla ) F'(\varphi) \de \f x \\
\leq{}& \| ( \nabla \tu )_{\sym,-}\|_{L^\infty(\Omega)} \| \nabla \varphi-\nabla \tp\|_{L^2(\Omega)}^2 + \| \Delta \tp - F'(\tp) \|_{W^{1,d}(\Omega) } \| \f u - \tu\|_{L^2(\Omega)} \|  \varphi -  \tp\|_{L^{2d/(d-2)}(\Omega)} \\
& +\int_\Omega (\tu\cdot \nabla )\tp \left (  F'(\varphi) - F'( \tp) - F''(\tp) ( \varphi - \tp) \right )  \de \f x 
\,.
\end{align*}
%
%
%
For the convection terms in the heat equation and the phase-field equation, we observe
\begin{align*}
\int_\Omega& ( \f u \cdot \nabla ) \tet ( f'_\delta (\theta ) - \varphi ) - \theta  ( \tu \cdot \nabla) ( f'_\delta (\tet) - \tp) + ( \f u \varphi )\cdot \nabla \tet - ( \tu \cdot \nabla) \tp  \theta \de \f x 
\\
={}& \int_\Omega ( \f u \cdot \nabla ) \tet ( f'_\delta (\theta ) - f'_\delta (\tet) ) - ( \tu \cdot \nabla) \tet    f''_\delta (\tet) ( \theta - \tet)   
 \de \f x 
\\
={}& \int_\Omega ( (\f u- \tu)  \cdot \nabla ) \tet ( f'_\delta (\theta ) -  f'_\delta  (\tet) )+ ( \tu \cdot \nabla ) \tet (  f'_\delta (\theta ) -  f'_\delta  (\tet)  -     f''_\delta (\tet) ( \theta - \tet))   \de \f x 
 \,.
\end{align*}
The first equation is valid due to a rearrangement using again the vanishing divergence of $\f u$ and $\tu$, \textit{i.e.}, the integrals over $(\f u\cdot \nabla) \tet f'_\delta (\tet) = ( \f u \cdot \nabla ) f_\delta (\tet) $ and 
$ \tet (\tu\cdot \nabla )f_\delta' (\tet) = \tet  f''_\delta (\tet) (\tu\cdot \nabla ) \tet = (\tu\cdot\nabla) \int_1^{\tet} r f''_\delta (r) \de r $
 vanish, respectively. The second equality follows again by adding and subtracting the appropriate terms. 
 Together, we find the estimate
 \begin{align*}
 \int_\Omega& ( \f u \cdot \nabla ) \tet ( f_\delta' (\theta ) - \varphi ) - \theta  ( \tu \cdot \nabla) ( f_\delta' (\tet) - \tp) + ( \f u \varphi)\cdot \nabla  \tet - ( \tu \cdot \nabla ) \tp \theta \de \f x 
\\
\leq{}& \| \nabla \tet\|_{L^\infty(\Omega) } \left ( \| \f u - \tu\|_{L^2(\Omega)}\left  \| f_\delta ' ( \theta ) - f_\delta ' ( \tet) \right \|_{L^2(\Omega)} \right ) \\&+
\|( \tu \cdot \nabla ) f'_\delta( \tet) \|_{L^\infty(\Omega)}  \int_\Omega  \left ( \theta -\tet - \left ( f''_\delta ( \tet) \right )^{-1} \left ( f'_\delta ( \theta ) - f_\delta'(\tet) \right )  \right ) \de \f x
 \,.
 \end{align*}
The remaining terms in~\eqref{firstest} including time derivatives, can be formally transformed to
\begin{align*}
 \t \tet& ( f_\delta' (\theta ) - \varphi) - \theta \t (  f_\delta' (\tet)-\tp) -\theta\t\tp 
 - \t \tet(  f_\delta' (\tet) - \tet  f_\delta''  (\tet) ) + \t \tet \varphi 
 \\
 ={}& \t \tet ( f_\delta' ( \theta) - f_\delta' ( \tet) ) - \t f_\delta '(\tet) ( \theta - \tet) 
 \\
 ={}&  -  \t f_\delta '(\tet)  \left ( \theta - \tet - \left ( f''_\delta ( \tet) \right )^{-1} \left ( f_\delta'(\theta ) - f'_\delta ( \tet) \right )\right )
 \,.
\end{align*}
Note that the two occurrences of $ \langle \t \varphi , \tet \rangle $ in~\eqref{firstest} already cancel each other. 
In the first equality all terms depending on $\varphi$ and $\tp$ vanish due to cancellations, the second equation follows by an application of the chain rule.
Consequently, we find 
\begin{align*}
\int_\Omega  \t \tet ( f_\delta'(\theta ) - \varphi) - \theta \t ( f_\delta'(\tet)-\tp) -\theta\t\tp - \t \tet( f_\delta'(\tet) - \tet f_\delta'' (\tet) ) + \t \tet \varphi
\de \f x \\
\leq  \|  \t f_\delta '(\tet) \|_{L^\infty}  \int_\Omega   \left ( \theta - \tet - \left ( f''_\delta ( \tet) \right )^{-1} \left ( f_\delta'(\theta ) - f'_\delta ( \tet) \right )\right ) \de \f x \,.
\end{align*}
Inserting everything back into~\eqref{firstest}, we may conclude
\begin{align}
\begin{split}
\mathcal{R}(\f u(t), \theta (t), \varphi (t) &|{}\tu(t),  \tet(t),\tp(t))  \\+ \int_0^t \int_\Omega& \left (  
\tet \frac{\nu(\theta)|(\nabla \f u)_{\sym} |^2 }{\theta} + \theta \frac{\nu(\tet) |( \nabla \tu)_{\sym}|^2}{\tet} - ( \nu(\theta ) + \nu(\tet) ) (\nabla \f u)_{\sym} : (\nabla \tu)_{\sym} 
    \right ) \de \f x  \de s \\
    + \int_0^t \int_\Omega& \left (  \tet\frac{| \nabla \mu|^2}{\theta}  + \theta \frac{| \nabla \tilde{\mu }|^2}{\tet} - 2 \nabla \mu \cdot \nabla \tilde{\mu}   \right )\de \f x  \de s \\
+ \int_0^t \int_\Omega& \left (   \tet  \kappa(\theta)    {| \nabla \log\theta|^2 } -    \kappa(\theta) {\nabla \log\theta }\cdot \nabla \tet +  \theta  \kappa(\tet)  {| \nabla\log \tet|^2} +    {\di (\kappa(\tet) \log \tet ) }   \theta   \right )\de \f x  \de s \\
\leq{}& \mathcal{R}(\f u_0 , \theta_0 , \varphi _0 | \tu_0 , \tet_0,\tp_0 ) \\&+ \int_0^t \int_\Omega \left ( \t \tp + (\tu\cdot \nabla )\tp\right ) \left (  F'(\varphi) - F'(\tp) - F'' (\tp) (\varphi - \tp)\right ) \de \f x \de s\\
& + \int_0^t \left \| (\nabla \tu )_{\sym} \right \|_{L^\infty(\Omega)}  \left (\| \f u -\tu\|_{L^2(\Omega)}^2 + \langle m , I \rangle \right ) \de s   
\\
&+\int_0^t  \| ( \nabla \tu )_{\sym}\|_{L^\infty(\Omega)} \| \nabla \varphi-\nabla \tp\|_{L^2(\Omega)}^2\de s \\ & +\int_0^t  \| \Delta \tp- F'(\tp)\|_{W^{1,d}(\Omega)} \| \f u - \tu\|_{L^2(\Omega)} c \| \nabla \varphi - \nabla \tp\|_{L^2(\Omega)}  \de s\\
&+  \int_0^t \| \nabla \tet\|_{L^\infty(\Omega) } \left ( \| \f u - \tu\|_{L^2(\Omega)}\left  \| f_\delta ' ( \theta ) - f_\delta ' ( \tet) \right \|_{L^2(\Omega)} \right )  \de s 
 \\
&+\int_0^t\|\t f_\delta '(\tet)  + ( \tu \cdot \nabla) f'_\delta( \tet) \|_{L^\infty(\Omega)}  \int_\Omega  \left ( \theta -\tet - \left ( f''_\delta ( \tet) \right )^{-1} \left ( f'_\delta ( \theta ) - f_\delta'(\tet) \right )  \right ) \de \f x \de s 
\\
&+ \frac{M}{2} \| \varphi - \tp \|_{(W^{1,\infty}(\Omega))^*} ^2 
\,.
 \end{split}\label{secondtest}
\end{align}

\subsection{Dissipative terms}
In this section, we consider the different dissipative terms arising in inequality~\eqref{firstest} and~\eqref{secondtest}.
Starting with the terms due to friction in the fluid, we observe by some manipulations that
\begin{align*}
\tet \frac{\nu(\theta)|(\nabla \f u)_{\sym} |^2 }{\theta}& + \theta \frac{\nu(\tet) |( \nabla \tu)_{\sym}|^2}{\tet} - ( \nu(\theta ) + \nu(\tet) ) (\nabla \f u)_{\sym} : (\nabla \tu)_{\sym} 
\\
={}& \nu(\theta ) \left ( \sqrt{\frac{\tet}{\theta}} ( \nabla \f u)_{\sym} - \sqrt{\frac{\theta}{\tet}} ( \nabla \tu )_{\sym} \right ) : \sqrt{\frac{\tet}{\theta}} ( \nabla \f u)_{\sym}
\\& - \nu(\tet) \left ( \sqrt{\frac{\tet}{\theta}} ( \nabla \f u)_{\sym} - \sqrt{\frac{\theta}{\tet}} ( \nabla \tu )_{\sym} \right ) :\sqrt{\frac{\theta}{\tet}} ( \nabla \tu )_{\sym} 
\\
={}& \nu(\theta ) \left | \sqrt{\frac{\tet}{\theta}} ( \nabla \f u)_{\sym} - \sqrt{\frac{\theta}{\tet}} ( \nabla \tu )_{\sym} \right |^2
\\& (\nu(\theta)- \nu(\tet)) \left ( \sqrt{\frac{\tet}{\theta}} ( \nabla \f u)_{\sym} - \sqrt{\frac{\theta}{\tet}} ( \nabla \tu )_{\sym} \right ) :\sqrt{\frac{\theta}{\tet}} ( \nabla \tu )_{\sym} \,.
%
%
%
\end{align*}
Applying Young's inequality in a standard manner, implies
\begin{align*}
\tet \frac{\nu(\theta)|(\nabla \f u)_{\sym} |^2 }{\theta}& + \theta \frac{\nu(\tet) |( \nabla \tu)_{\sym}|^2}{\tet} - ( \nu(\theta ) + \nu(\tet) ) (\nabla \f u)_{\sym} : (\nabla \tu)_{\sym} 
\\
\geq{}&  \nu (\theta)\frac{1}{2} \left | \sqrt{\frac{\tet}{\theta}} ( \nabla \f u)_{\sym} - \sqrt{\frac{\theta}{\tet}} ( \nabla \tu )_{\sym} \right |^2 - \frac{(\nu(\theta)- \nu(\tet))^2}{2\nu(\theta)} \frac{\theta}{\tet} |( \nabla \tu)_{\sym}|^2\,.
\end{align*}
Similar, but somehow simpler we find for the dissipative terms due to the chemical potential after some manipulations that
\begin{align*}
 \tet\frac{| \nabla \mu|^2}{\theta}  + \theta \frac{| \nabla \tilde{\mu }|^2}{\tet} - 2 \nabla \mu \cdot \nabla \tilde{\mu} = \left  | \sqrt\frac{\tet}{\theta} \nabla \mu - \sqrt\frac{\theta}{\tet}\nabla \tmu\right | ^2 
\,.
 \end{align*}

Concerning the terms due to the heat conduction, we first consider the case $\kappa ( \theta) = \kappa_0$, thus $\beta=0$ and also $\delta =0$. 
Note that $$\int_\Omega \tet \kappa (\tet) | \nabla \ln \tet |^2 + \tet \di (\kappa(\tet) \nabla \ln \tet ) \de \f x = \int_\Omega \di ( \kappa(\tet) \nabla \tet ) \de \f x = 0\,.$$
We observe with some algebraic transformations that
\begin{align*}
  \tet &   {| \nabla \log\theta|^2 } -    {\nabla \log\theta }\cdot \nabla \tet + ( \theta - 
\tet )  {| \nabla\log \tet|^2} +    {\Delta\log \tet } (  \theta -  \tet) \\
 &=     \tet  \nabla \log \theta\cdot (  \nabla \log \theta- \nabla \log \tet ) + ( \theta - \tet ) | \nabla \log \tet |^2 + \Delta \log \tet   ( \theta - \tet )   \\
 &= \tet \left ( | \nabla \log \theta - \nabla \log \tet |^2 + \nabla \log \tet \cdot(\nabla \log \theta - \nabla \log \tet ) \right ) \\ 
 &\quad + \left ( \theta - \tet - \tet ( \log \theta - \log \tet ) | \nabla \log \tet |^2  + \tet ( \log \theta - \log \tet ) | \nabla \log \tet|^2 \right )\\
& \quad +  \Delta \log \tet  ( \theta -\tet  - \tet ( \log \theta - \log \tet ))   + \Delta \log \tet    ( \tet ( \log \theta - \log \tet ) ) \,.
 \end{align*}
From an integration-by-parts on the last term, using the fact that $ \nabla \log \tet \cdot \f n = ( \nabla \tet \cdot \f n ) / \tet = 0$ vanishes on the boundary (see~\eqref{boundary}),  and the product rule, we may infer
\begin{align*}
 \int_\Omega \tet \nabla \log \tet \cdot(\nabla \log \theta - \nabla \log \tet ) + \tet ( \log \theta - \log \tet ) | \nabla \log \tet|^2\de \f x - \int_\Omega \nabla \log \tet \cdot  \nabla ( \tet ( \log \theta - \log \tet ) )\de \f x  = 0\,.
\end{align*}
We may conclude that
\begin{align*}
\kappa_0  \int_\Omega  \tet &   {| \nabla \log\theta|^2 } -    {\nabla \log\theta }\cdot \nabla \tet + ( \theta - 
\tet )  {| \nabla\log \tet|^2} +    {\Delta\log \tet } (  \theta -  \tet) 
\de \f x    \\
&= \kappa_0 \int_{\Omega}  \tet | \nabla \log \theta - \nabla \log \tet |^2 \de \f x  
\\
&\quad +\kappa_0 \int_{\Omega}  ( \theta - \tet - \tet ( \log \theta - \log \tet ) ) ( | \nabla \log \tet|^2 + \Delta  \log \tet ) \de \f x  
\,.
\end{align*}

For $ \beta \in (0,2]$ and $ 2\delta \leq \beta$, we find
\begin{align}\label{number}
\begin{split}
\int_\Omega \tet \theta^\beta & | \nabla \log \theta |^2 - \theta^\beta  \nabla \tet \nabla \log \theta + (\theta-\tet) \tet^\beta  | \nabla \log \tet |^2  + (\theta - \tet) \di ( \tet^\beta \nabla \log \tet)\de \f x  \\
={}& \frac{4}{\beta^2 }\int_\Omega \left ( \tet | \nabla {\theta}^{\beta/2}|^2 - \tet ^{1-\beta/2}{\theta}^{\beta/2}  \nabla {\tet}^{\beta/2}\nabla \theta^{\beta/2}  \right ) \de \f x 
\\&+ \frac{4}{\beta^2 }\int_\Omega \left ( ( \theta-\tet)  | \nabla {\tet}^{\beta/2}|^2 + \frac{\beta}{2} (\theta - \tet) \di ( \tet^{\beta/2} \nabla  \tet^{\beta/2}) \right ) \de \f x \
\,.
%
\end{split}
\end{align}
For the first line on the right-hand side of~\eqref{number}, we observe
\begin{align*}
\tet& | \nabla {\theta}^{\beta/2}|^2 - \tet ^{1-\beta/2}{\theta}^{\beta/2}  \nabla {\tet}^{\beta/2}\nabla \theta^{\beta/2}  
\\
= {}& \tet | \nabla {\theta}^{\beta/2}  - \nabla {\tet}^{\beta/2} |^2+ \tet ^{1-\beta/2}( \tet ^{\beta/2}- {\theta}^{\beta/2} )  \nabla {\tet}^{\beta/2}\nabla \theta^{\beta/2} + \tet \nabla {\tet}^{\beta/2} ( \nabla {\theta}^{\beta/2}  - \nabla {\tet}^{\beta/2}) 
\\
= {} & \tet | \nabla {\theta}^{\beta/2}  - \nabla {\tet}^{\beta/2} |^2+ \tet ^{1-\beta/2}( \tet ^{\beta/2}- {\theta}^{\beta/2} )  \nabla {\tet}^{\beta/2}(\nabla \theta^{\beta/2}- \nabla {\tet}^{\beta/2}) 
\\ &  + \tet \nabla {\tet}^{\beta/2} ( \nabla {\theta}^{\beta/2}  - \nabla {\tet}^{\beta/2})+  \tet ^{1-\beta/2}( \tet ^{\beta/2}- {\theta}^{\beta/2} ) | \nabla {\tet}^{\beta/2}|^2 \,,
\end{align*}
and for the second line on the right-hand side of~\eqref{number}, we observe 
\begin{align*}
( \theta& -\tet)  | \nabla {\tet}^{\beta/2}|^2 + \frac{\beta}{2} (\theta - \tet) \di ( \tet^{\beta/2} \nabla  \tet^{\beta/2})
\\
={}& \left ( \theta - \tet -\left  ( f_\delta''(\tet)\right )^{-1}\left ( f'_\delta(\theta) - f'_\delta ( \tet ) \right ) 
\right ) (| \nabla {\tet}^{\beta/2}|^2 + \frac{\beta}{2} \di ( \tet^{\beta/2} \nabla  \tet^{\beta/2})) 
\\
&+\left  ( f_\delta''(\tet)\right )^{-1}\left ( f'_\delta(\theta) - f'_\delta ( \tet ) \right )  (| \nabla {\tet}^{\beta/2}|^2 + \frac{\beta}{2} \di ( \tet^{\beta/2} \nabla  \tet^{\beta/2}))
\\
={}&\left  ( \theta - \tet -\left  ( f_\delta''(\tet)\right )^{-1}\left ( f'_\delta(\theta) - f'_\delta ( \tet ) \right ) \right  ) (| \nabla {\tet}^{\beta/2}|^2 + \frac{\beta}{2} \di ( \tet^{\beta/2} \nabla  \tet^{\beta/2})) 
\\
&+\tet^{\beta/2-1}\left  ( f_\delta''(\tet)\right )^{-1}\left ( f'_\delta(\theta) - f'_\delta ( \tet ) \right )   \frac{\beta}{2} (\tet^{1-\beta/2}| \nabla {\tet}^{\beta/2}|^2 + \di ( \tet \nabla  \tet^{\beta/2}))\,.
\end{align*}
To combine the previous two equations, we observe that the two second lines on the right-hand sides may be related via an integration-by-parts
\begin{align*}
\int_\Omega \tet \nabla {\tet}^{\beta/2} ( \nabla {\theta}^{\beta/2}  - \nabla {\tet}^{\beta/2}) \de \f x =  \int_\Omega ( \tet ^{\beta/2}- {\theta}^{\beta/2} )  \di ( \tet \nabla  \tet^{\beta/2}) \de \f x 
\end{align*}
and the algebraic relation
\begin{align*}
\tet ^{1-\beta/2}&( \tet ^{\beta/2}- {\theta}^{\beta/2} ) | \nabla {\tet}^{\beta/2}|^2 +  ( \tet ^{\beta/2}- {\theta}^{\beta/2} )  \di ( \tet \nabla  \tet^{\beta/2}) 
\\
&+\tet^{\beta/2-1} \left  ( f_\delta''(\tet)\right )^{-1}\left ( f'_\delta(\theta) - f'_\delta ( \tet ) \right )   \frac{\beta}{2} (\tet^{1-\beta/2}| \nabla {\tet}^{\beta/2}|^2 + \di ( \tet \nabla  \tet^{\beta/2}))
\\
={}& - (\tet^{1-\beta/2}| \nabla {\tet}^{\beta/2}|^2 + \di ( \tet \nabla  \tet^{\beta/2})) \left (  {\theta}^{\beta/2}  - \tet ^{\beta/2} -\frac{\beta}{2}  \tet^{\beta/2-1} \left  ( f_\delta''(\tet)\right )^{-1}\left ( f'_\delta(\theta) - f'_\delta ( \tet ) \right )   \right )
\\
={}& - ( \tet^{\beta/2} \di ( \tet^{1-\beta/2} \nabla  \tet^{\beta/2})) \left (  {\theta}^{\beta/2}  - \tet ^{\beta/2} -\frac{\beta}{2}  \tet^{\beta/2-1} \left  ( f_\delta''(\tet)\right )^{-1}\left ( f'_\delta(\theta) - f'_\delta ( \tet ) \right )   \right )
\,.
\end{align*}
Taking everything together into~\eqref{number}, we may conclude by Young's inequality that
\begin{align*}
\int_\Omega \tet \theta^\beta & | \nabla \log \theta |^2 - \theta^\beta  \nabla \tet \nabla \log \theta + (\theta-\tet) \tet^\beta  | \log \tet |^2  + (\theta - \tet) \di ( \tet^\beta \nabla \log \tet)\de \f x  \\
={}&\frac{4}{\beta^2 } \int_\Omega  \tet | \nabla {\theta}^{\beta/2}  - \nabla {\tet}^{\beta/2} |^2+ \tet ^{1-\beta/2}( \tet ^{\beta/2}- {\theta}^{\beta/2} )  \nabla {\tet}^{\beta/2}(\nabla \theta^{\beta/2}- \nabla {\tet}^{\beta/2})  
\de \f x 
\\& +\frac{4}{\beta^2 } \int_\Omega\left  ( \theta - \tet -\left  ( f_\delta''(\tet)\right )^{-1}\left ( f'_\delta(\theta) - f'_\delta ( \tet ) \right ) \right  ) (| \nabla {\tet}^{\beta/2}|^2 + \frac{\beta}{2} \di ( \tet^{\beta/2} \nabla  \tet^{\beta/2}))  \de \f x 
\\ &-\frac{4}{\beta^2 } \int_\Omega  (\tet^{\beta/2} \di ( \tet^{1-\beta/2} \nabla  \tet^{\beta/2})) \left (  {\theta}^{\beta/2}  - \tet ^{\beta/2} -\frac{\beta}{2}  \tet^{\beta/2-1}\left  ( f_\delta''(\tet)\right )^{-1}\left ( f'_\delta(\theta) - f'_\delta ( \tet ) \right )  \right )\de \f x \\
\geq {}&\frac{2}{\beta^2 } \int_\Omega  \tet | \nabla {\theta}^{\beta/2}  - \nabla {\tet}^{\beta/2} |^2 -  \tet ^{1-\beta}| \tet ^{\beta/2}- {\theta}^{\beta/2} |^2|  \nabla {\tet}^{\beta/2} |^2 \de \f x 
\\& +\frac{4}{\beta^2 } \int_\Omega\left  ( \theta - \tet -\left  ( f_\delta''(\tet)\right )^{-1}\left ( f'_\delta(\theta) - f'_\delta ( \tet ) \right ) \right  ) (| \nabla {\tet}^{\beta/2}|^2 + \frac{\beta}{2} \di ( \tet^{\beta/2} \nabla  \tet^{\beta/2}))  \de \f x 
\\ &-\frac{4}{\beta^2 } \int_\Omega  (\tet^{\beta/2} \di ( \tet^{1-\beta/2} \nabla  \tet^{\beta/2})) \left (  {\theta}^{\beta/2}  - \tet ^{\beta/2} -\frac{\beta}{2}  \tet^{\beta/2-1}\left  ( f_\delta''(\tet)\right )^{-1}\left ( f'_\delta(\theta) - f'_\delta ( \tet ) \right )  \right )\de \f x 
\,.
\end{align*}
To handle the difference of the temperatures in the $L^2$-norm for~$\beta>1-\delta $, we need to absorb some parts into the dissipative terms. 
Via an Gagliardo Nirenberg inequality, we observe
\begin{align*}
\| {\theta}^{\beta/2}  - \tet ^{\beta/2}\|_{L^2}^2 \leq{}& c \left ( \| \nabla {\theta}^{\beta/2}  - \nabla \tet ^{\beta/2} \|_{L^2}^{2d/(d+2)} \| {\theta}^{\beta/2}  - \tet ^{\beta/2} \|_{L^1}^{4/(d+2)} + \|{\theta}^{\beta/2}  - \tet ^{\beta/2}\|_{L^1}^2 \right ) \\
\leq {}& \frac{1}{\beta^2} \|  \nabla {\theta}^{\beta/2}  - \nabla \tet ^{\beta/2} \|_{L^2} ^2 + C_\varepsilon \| {\theta}^{\beta/2}  - \tet ^{\beta/2} \|_{L^1}^2 \,.
\end{align*}
From Lemma~\ref{lem:diff} and Lemma~\ref{lem:relpos}, we find for $\beta \in (4\delta , 2-2\delta]$  that
\begin{align*}
\int_\Omega \tet \theta^\beta & | \nabla \log \theta |^2 - \theta^\beta  \nabla \tet \nabla \log \theta + (\theta-\tet) \tet^\beta  | \log \tet |^2  + (\theta - \tet) \di ( \tet^\beta \nabla \log \tet)\de \f x  \\
\geq {}&\frac{1}{\beta^2 } \int_\Omega  \tet | \nabla {\theta}^{\beta/2}  - \nabla {\tet}^{\beta/2} |^2 \de \f x  
 -c  \int_\Omega \Lambda_\delta ( \theta | \tet)  \left ( \left \|  \nabla {\tet}^{\beta/2}\right \|^2_{L^\infty(\Omega)} + \| \Delta   \tet^{\beta/2}\|_{L^\infty(\Omega)} \right  )  \de \f x \,.
\end{align*}

\subsection{Nonconvex contribution}
This section concerns the last term on the right-hand side of~\eqref{firstest}.
\begin{align*}
\frac{1}{2} \t \| \varphi - \tp \|_{(W^{1,\infty}(\Omega))^*} ^2 =  \| \varphi - \tp \|_{(W^{1,\infty}(\Omega))^*}  \sup_{\| \Phi \|_{W^{1,\infty}}=1} \left \langle \t \varphi - \t \tp , \Phi \right \rangle \,.
\end{align*}
Since the equation~\eqref{phaseeq} holds, we may find
\begin{align}
\begin{split}
 \left \langle \t \varphi - \t \tp , \Phi \right \rangle ={}& -\int_\Omega (\f u \varphi - \tu \tp ) \cdot 
\nabla \Phi + (\nabla \mu -\nabla \tilde \mu ) \cdot \nabla \Phi \de \f x 
\\
={}&- \int_\Omega ((\f u-\tu)  \varphi + \tu(\varphi - \tp) ) \cdot 
\nabla \Phi 
\de \f x \\
&
- \int_\Omega \left (
\sqrt\frac{\tet}{{\theta}} \nabla \mu- \sqrt\frac{\theta}{{\tet}} \nabla \tilde \mu
\right )
\sqrt{
\frac
{\theta}
{\tet} 
}
 \cdot \nabla \Phi+ \nabla \tilde{ \mu} \left (\frac{{\theta }}{{\tet}}-1\right ) \cdot \nabla \Phi \de \f x 
\\
\leq{}&\left ( \| \f u -\tu \|_{L^2(\Omega)} \| \varphi \|_{L^2(\Omega)} + \| \varphi - \tp \|_{L^2(\Omega)} \| \tu \|_{L^2(\Omega)}  \right ) \| \nabla \Phi\|_{L^\infty(\Omega)} \\&+\left ( \left  \| \sqrt\frac{\tet}{{\theta}} \nabla \mu-\sqrt \frac{\theta}{{\tet}} \nabla \tilde \mu \right \|_{L^2(\Omega)} \left \| \frac{\theta}{\tet} \right \|_{L^1(\Omega)} ^{1/2} + \left \| \frac{\nabla \tilde\mu }{{\tet}}\right \|_{L^\infty (\Omega)} \| {\theta} -{\tet} \|_{L^1(\Omega)} \right ) \| \nabla \Phi\|_{L^\infty(\Omega)}\,.
\end{split}\label{nonconvex}
\end{align}
Since $\varphi$ and $\tp$ have the same mean, the Poincar\'e inequality holds for its difference
since $\int_\Omega \varphi - \tp \de \f x= 0$, \textit{i.e.}, 
\begin{align*}
\| \varphi - \tp \|_{L^2(\Omega)} \leq c \| \nabla \varphi - \nabla \tp \|_{L^2(\Omega)} \,.
\end{align*}
such that
\begin{align*}
\frac{1}{2} \t \| \varphi - \tp \|_{(W^{1,\infty}(\Omega))^*} ^2 \leq{}&   \| \varphi - \tp \|_{(W^{1,\infty}(\Omega))^*}   \left ( 
\| \f u -\tu \|_{L^2(\Omega)} \| \varphi \|_{L^2(\Omega)}+ c \| \nabla \varphi -\nabla  \tp \|_{L^2(\Omega)} \| \tu \|_{L^2(\Omega)}\right ) \\&+  \| \varphi - \tp \|_{(W^{1,\infty}(\Omega))^*}     \left  \| \sqrt\frac{\tet}{{\theta}} \nabla \mu- \sqrt\frac{\theta}{{\tet}} \nabla \tilde \mu \right \|_{L^2(\Omega)} \left \| \frac{\theta}{\tet} \right \|_{L^1(\Omega)} ^{1/2} 
\\&+  \| \varphi - \tp \|_{(W^{1,\infty}(\Omega))^*}   \left \| \frac{\nabla \tilde\mu }{{\tet}}\right \|_{L^\infty (\Omega)} \| {\theta} -{\tet} \|_{L^1(\Omega)} 
 \,.
\end{align*}


Combining all the estimates, we conclude that
\begin{align*}
\mathcal{R}&(\f u(t), \theta (t), \varphi (t) |{}\tu(t),  \tet(t),\tp(t))
+ \int_0^t { \kappa_0} \int_{\Omega}  \tet | \nabla \log \theta - \nabla \log \tet |^2 \de \f x  \\
+& \int_0^t \int_\Omega  \nu (\theta)\frac{1}{2} \left | \sqrt{\frac{\tet}{\theta}} ( \nabla \f u)_{\sym} - \sqrt{\frac{\theta}{\tet}} ( \nabla \tu )_{\sym} \right |^2 \de \f x  \de s
\\+&\int_0^t \int_\Omega \frac{2\kappa_1}{\beta^2 }  \tet | \nabla {\theta}^{\beta/2}  - \nabla {\tet}^{\beta/2} |^2 +    \left | \sqrt\frac{\tet}{\theta}\nabla \mu - \sqrt\frac{\theta}{\tet}\nabla \tmu \right |^2 \de \f x  \de s 
 \\
\leq{}&\mathcal{R}(\f u_0 , \theta_0 , \varphi _0 | \tu_0 , \tet_0,\tp_0 ) + \int_0^t \int_\Omega(\t \tp+ ( \tu \cdot \nabla) \tp)  \left (  F'(\varphi) - F'(\tp) - F'' (\tp) (\varphi - \tp)\right ) \de \f x \de s\\
& +  \int_0^t \left \| (\nabla \tu )_{\sym,-} \right \|_{L^\infty(\Omega)}  \left (\| \f u -\tu\|_{L^2(\Omega)}^2 + \langle m , I \rangle\right )  \de s   
\\
&+\int_0^t  \| ( \nabla \tu )_{\sym,-}\|_{L^\infty(\Omega)} \| \nabla \varphi-\nabla \tp\|_{L^2(\Omega)}^2 \de s \\ &+c\int_0^t \| \Delta \tp- F'(\tp)\|_{W^{1,d}(\Omega)} \left ( \| \f u - \tu\|_{L^2(\Omega)	} ^2 + \| \nabla \varphi - \nabla \tp\|_{L^2(\Omega)} ^2 \right ) \de s
 \\
& +\int_0^t\|\t f_\delta '(\tet)  + ( \tu \cdot \nabla) f'_\delta( \tet) \|_{L^\infty(\Omega)} \int_\Omega   \left ( \theta - \tet - \left ( f''_\delta ( \tet) \right )^{-1} \left ( f_\delta'(\theta ) - f'_\delta ( \tet) \right )\right ) \de \f x\de s 
\\
\\& +c \int_0^t \int_\Omega \Lambda_\delta ( \theta | \tet)  \left ( \left \|  \nabla {\tet}^{\beta/2}\right \|^2_{L^\infty(\Omega)} + \| \Delta   \tet^{\beta/2}\|_{L^\infty(\Omega)} \right  )  \de \f x \de s 
\\&+M \int_0^t
\| \varphi - \tp \|_{(W^{1,\infty}(\Omega))^*}   \left ( 
\| \f u -\tu \|_{L^2(\Omega)} \| \varphi \|_{L^2(\Omega)}+ c \| \nabla \varphi -\nabla  \tp \|_{L^2(\Omega)} \| \tu \|_{L^2(\Omega)}\right ) \de s 
\\
&+M  \int_0^t  \| \varphi - \tp \|_{(W^{1,\infty}(\Omega))^*}   \left (  \left  \| \sqrt\frac{\tet}{{\theta}} \nabla \mu-\sqrt \frac{\theta}{{\tet}} \nabla \tilde \mu \right \|_{L^2(\Omega)} \left \| \frac{\theta}{\tet} \right \|_{L^1(\Omega)} ^{1/2} + \left \| \frac{\nabla \tilde\mu }{{\tet}}\right \|_{L^\infty (\Omega)} \| {\theta} -{\tet} \|_{L^1(\Omega)} 
\right ) \de s 
\\
&+  \int_0^t \frac{1}{2} \| \nabla \tet\|_{L^\infty (\Omega)} \left ( \| \f u - \tu\|_{L^2(\Omega)}^2+ \left  \| f_\delta ' ( \theta ) - f_\delta ' ( \tet) \right \|_{L^2(\Omega)} ^2 \right )  + \int_\Omega \frac{(\nu(\theta)- \nu(\tet))^2}{2} \frac{\theta}{\tet} |( \nabla \tu)_{\sym}|^2\de \f x \de s 
\,.
\end{align*}
The last line may be estimated by the relative energy due to Lemma~\ref{lem:log} or Lemma~\ref{lem:diff} and Lemma~\ref{lem:fenchel}. 
The last term in the second to the last line, we may estimate by $$ \| \theta - \tet \| _{L^1(\Omega)} \leq \int_\Omega \theta - \tet - ( f_\delta ''(\tet))^{-1} ( f'_\delta (\theta )- f_\delta '(\tet)) \de \f x + c \| f'_\delta ( \theta ) - f'_\delta (\tet)\|_{L^1(\Omega)} \,,$$
which can be further estimated by Lemma~\ref{lem:relpos} and~\ref{lem:log}.  
From the calculation in \cite[Section 4.3]{larosc}, we find for a function fulfilling Hypothesis~\ref{hypo} that
\begin{multline}\label{st82}
 \int_0^t \int_\Omega | \tp_t + ( \tu \cdot \nabla) \tp | (F'' ( \tp ) ( \tp - \varphi) - F' ( \tp ) + F' ( \varphi)) \de \f x  \de s\\
  \le c \int_0^t \| \tp_t+ ( \tu \cdot \nabla) \tp \|_{L^\infty(\Omega)} \mathcal{E}( \f u, \theta , \varphi |\tu,  \tet , \tp) \de s\,.
\end{multline}

Applying Gronwall's estimate implies the relative energy inequality and thus Proposition~\ref{prop:rel}. 
%
\begin{proof}[Proof of Theorem~\ref{thm:diss}]
In order to prove Theorem~\ref{thm:diss} we have to show the relative energy inequality~\eqref{relenenergy} for every test function $(\tu,\tet,\tp)\in\mathcal{Y}$, which is not assumed to be a solution anymore. 
This can be done by adapting the proof of the previous section and adding and simultaneously subtracting the equations for~$(\tu,\tet,\tp)\in\mathcal{Y}$ in~\eqref{energyestimates},~\eqref{momentumtest},~\eqref{entrotestzwei}~\eqref{strongeqphi}, and~\eqref{nonconvex}. 
This gives rise to the solution operator $\mathcal{A}$ defined in~\eqref{operatorA} such that we infer
\begin{multline*}
\mathcal{R}(\f q  |\tilde {\f q}) (t) + \frac{1}{2}\langle m_t , I \rangle + \int_0^t \mathcal{W} (\f q  |\tilde {\f q})  e^{\int_s^t  \mathcal{K} (\tilde {\f q}) \de \tau } \de s 
\leq 
\mathcal{R}(\f q  |\tilde {\f q}) (0) e^{\int_0^t  \mathcal{K} (\tilde {\f q}) \de s }
\\+ \int_0^t \left (  \left \langle \mathcal{A}(\tilde {\f q})  , \begin{pmatrix}
\tu - \f u \\\tet-\theta
 \\ \tilde{\mu}-\mu 
 \end{pmatrix} \right \rangle  +  M \| \varphi - \tp \|_{( W^{1,\infty}(\Omega))^*} \| \mathcal{A}_3(\tilde {\f q})    \|_{( W^{1,\infty}(\Omega))^*}  \right )  e^{\int_s^t  \mathcal{K} (\tilde {\f q}) \de \tau } \de s \,.
 \end{multline*}
Observing that $\langle m_t , I \rangle$ is non-negative and may be estimated from below by zero, this implies the inequality~\eqref{relenenergy} and thus the assertion.

\end{proof}

\small
\bibliographystyle{abbrv}

\end{document}